\documentclass[11pt, reqno]{article}
\usepackage{amsmath, amsfonts, amssymb, mathtools}
\usepackage{alphabeta}
\usepackage{graphicx, amsthm, color}
\usepackage{pgfplots, caption, subcaption, tikz, tikz-cd} 
\usepackage{float}
\usepackage{enumitem}

\usepackage{xurl} 

\usepackage{cleveref}
\usepackage{float}

\usepackage{tikz-cd}
\usepackage{tikz}
\usetikzlibrary{arrows.meta, decorations.pathmorphing}

\everymath{\displaystyle}

\usepackage{euscript}
\usepackage[scr=boondox]{mathalpha}
\usepackage[sans]{dsfont}

\usepackage[T1]{fontenc}
\usepackage{fbb}

\newtheorem{theorem}{Theorem}[section]
\newtheorem{definition}[theorem]{Definition}
\newtheorem{lemma}[theorem]{Lemma}
\newtheorem{proposition}[theorem]{Proposition}
\newtheorem{corollary}[theorem]{Corollary}

\newtheorem{claim}[theorem]{Claim}

\theoremstyle{definition}

\numberwithin{equation}{section}
\newtheorem{remark}[theorem]{Remark}

\newcommand{\R}{\mathbb{R}}
\newcommand{\Sphere}{\mathbb{S}^{n - 1}}
\newcommand{\vtil}{\widetilde{\mathrm{V}}}
\newcommand{\esssup}{\operatorname*{ess\,sup}}
\newcommand{\dila}{\mathscr{d}}
\newcommand{\Sset}{\mathscr S}

\begin{document}

\begin{center}
{\LARGE Power--law asymptotics of \\ fractional $L^p$ polar projection bodies }
\end{center}
\smallskip
\begin{center}
{\Large \textsc{Tr\'i Minh L\^e}}
\end{center}

\bigskip

\noindent\textbf{Abstract.} 
The notion of $s$--fractional $L^p$ polar projection bodies, recently introduced by Haddad and Ludwig (Math.\ Ann.\ \textbf{388}:1091--1115, 2024), provides a bridge between fractional Sobolev theory and convex geometry.
In this manuscript, we study the limit of their Minkowski gauges under two natural asymptotic regimes: 
\\
\hspace*{3em} (a) first sending $p \to \infty$ and then $s \to 1^-$; \\ 
\hspace*{3em} (b) first sending $s \to 1^-$ and then $p \to \infty$. \\
Our main result shows that these two limiting processes commute.
As a consequence, we derive precise asymptotic behavior for the associated volumes and dual mixed volumes, thereby linking the (fractional) $L^p$ polar projection bodies to the newly introduced (fractional)  $L^\infty$ polar projection bodies.
These results further yield new geometric inequalities, including endpoint Lipschitz/Hölder isoperimetric--type and variants of Pólya--Szeg\H{o} inequalities in the $L^\infty$ setting.
\bigskip


\vspace{0.6cm}

\textit{Secondary} 35R11, 26D10


\section{Introduction}

Fix $p \in [1, + \infty)$ and $f \in W^{1, p}(\R^n)$.
The \textit{$L^p$ polar projection body $\Pi^\ast_p \, f$} is a star body whose gauge function is determined by
	\begin{align*}
		\| \xi \|_{\Pi^\ast_p \, f}^p := \int_{\R^n} |\langle \nabla f(x), \xi \rangle|^p \, dx \quad \text{ for every } \xi \in \R^n.
	\end{align*}
These star bodies have been used to translate affine Sobolev inequalities into geometric ones, which also provides a more natural understanding of the equality case.
They therefore capture the geometric essence behind the affine Sobolev inequality, see e.g \cite{HS_2009, LYZ_2002, Talenti_1976, Zhang_1999}.
Motivated by this geometric connection, Haddad and Ludwig in ~\cite{HL_2024,HL_2025} recently introduced the \textit{$s$--fractional $L^p$ polar projection bodies} $\Pi^{\ast, s}_p \, f$ with $s \in (0, 1)$, whose gauge function is defined by
\begin{align*}
    \| \xi \|^{sp}_{\Pi^{\ast, s}_p \, f} := p(1 - s) \int_0^\infty t^{-sp - 1} \int_{\R^n} |f(x + t\xi) - f(x)|^p \, dx \, dt \quad \text{ for every } \xi \in \R^n.
\end{align*}
This star body plays a fundamental role in transferring functional inequalities into a geometric setting, as it allows the affine fractional $L^p$ Sobolev inequalities to be expressed equivalently as geometric inequalities.
They also showed that, at the Minkowski gauge level, the limit of $\Pi^{\ast, s}_p \, f$ as $s \to 1^-$ is precisely $\Pi^\ast_p \, f$, thereby establishing a bridge between the fractional $L^p$ polar projection bodies and their classical $L^p$ counterparts.
A key motivation comes from the Bourgain--Brezis--Mironescu formula~\cite{BBM_2001}, which connects the Gagliardo seminorm and the $L^p$ norm of the gradient:
	\begin{align*}
		\lim_{s \to 1^-} p(1 - s) \int_{\R^n} \int_{\R^n} \dfrac{|f(x) - f(y)|^p}{|x - y|^{n + ps}} \, dx \, dy = \alpha_{n, p} \int_{\R^n} | \nabla f(x) |^p dx \quad \text{for every $f \in W^{1, p}(\R^n)$,}
	\end{align*}
where the constant $\alpha_{n, p}$ is determined via
	\[
		\alpha_{n, p} = \int_{\Sphere} | \langle \xi, \eta \rangle |^p \, d\xi, \text{ for every } \eta \in \Sphere.
	\]
Beyond this asymptotic regime, the limit of Gagliardo seminorm as $s \to 0^+$ (under an appropriate scaling), the asymmetric/anisotropic analogues and similar observations in the framework of $\Gamma$--convergence have also been observed, see~\cite{ADM_2011, CLKNP_2023, L_2014, M_2014, MS_2002}.

\medskip

Studying the limit of functionals and equations in $L^p$ spaces as $p \to \infty$ is ubiquitous in Calculus of Variations and PDEs, in which, at the limit, it often reveals extremal structures and degenerate equations (infinity Laplace equations), see e.g. \cite{ACJ_2004, Yu_2006}.
In the fractional setting, it is proved that solutions of the fractional $p$--Laplace equation converge to solutions of the so--called H\"older infinity Laplace equation, see~\cite{CLM_2012}.

\medskip

Furthermore, the interchangeability of limits is a classical and important issue in analysis, e.g Moore--Osgood theorem.
In the last twenty years, such phenomena have also attracted considerable attention in \emph{homogenization theory}, where one often studies the interplay between two limiting processes---the small-scale parameter and the exponent $p$ in $L^p$--in order to capture the correct asymptotic behavior of local/nonlocal integral functionals via $\Gamma$--convergence, see e.g~\cite{BN_2008,  BPG_2004, CP_2005, DEZ_2024, KZ_2020}.
For instance, one often considers variational energies of the form (under suitable assumptions on $\Omega$ and $H$)
\[
f \mapsto \int_\Omega H(x/\varepsilon, \nabla f(x))^p\,dx
\]
and observes whether the limits $\varepsilon\to0$ and $p\to\infty$ commute.
Analogous phenomena have recently been investigated in the context of optimal transportation, see~\cite{BCP_2024}.

\medskip

In light of these examples, it is natural to ask whether similar exchanging limit of two parameters phenomena arise in the setting of $s$--fractional $L^p$ polar projection bodies.
If we fix $f$ in $W^{1, 1}(\R^n)~ \cap~ W^{1, \infty}(\R^n)$, it is straightforward to see that the $L^p$ norm of $\nabla f$ converges to $L^\infty$ norm of $\nabla f$.
This simple observation, together with Haddad--Ludwig's result, shows that the process of taking first the limit as $s \to 1^-$ and then $p \to \infty$ for $\Pi^{\ast, s}_p \, f$ is reasonably well understood in isolation; however, their combined behavior has not yet been well investigated.
As we have said, such $''$double--limit$''$ problems often exhibit subtle interactions--in general, interchanging the order of limits may lead to different values.
This motivates the guiding question of the present work:
\begin{center}
\textit{
   Can the limit process of \( \Pi^{\ast, s}_p \, f \) as \( s \to 1^- \) and \( p \to \infty \) be interchanged? \\ \smallskip 
   What are the respective limits?
}
\end{center}

\medskip

Resolving this double-limit problem is not merely a technical refinement. 
The interaction between the parameters $s$ and $p$ encodes a structural rigidity of fractional $L^p$ polar projection bodies, ensuring that the theory extends coherently to the $L^\infty$ setting. 
In this sense, establishing the commutativity of the two limiting processes represents a necessary advance: it completes the asymptotic picture and provides the foundation for geometric inequalities in the anisotropic $L^\infty$ setting. 

\medskip

\textbf{Main contributions}.

\medskip

Our first contribution investigates the limits~(2),~(3), and~(4) in Figures~\ref{fig:two-limit-diagrams}, complementing the result of~\cite{HL_2024, HL_2025} which established limit~(1) in Figure \ref{fig:two-limit-diagrams} and \ref{fig:vtil^1/p}.

\medskip

For a fixed function $f \in W^{1, 1}(\R^n) \cap W^{1, \infty}(\R^n)$, concerning the $s$--fractional $L^p$ polar projection body $\Pi^{\ast, s}_p \, f$, our first main result establishes that the two limiting procedures:
\[
\,^{\prime\prime} \text{ first } p \to \infty  \text{ then } s \to 1^{-} \,^{\prime\prime}
\qquad \text{ and } \qquad
\,^{\prime\prime} \text{ first } s \to 1^{-} \text{ then } p \to \infty \,^{\prime\prime}
\]
\emph{commute} when evaluated at the level of Minkowski gauges. 
This compatibility result connects the fractional \( L^p \) polar projection bodies \( \Pi^{\ast, s}_p f \) (respectively their $L^p$ counterparts \( \Pi^\ast_p f \) ) and the newly introduced \( L^\infty \) variants \( \Pi^{\ast, s}_\infty f \) (respectively \( \Pi^\ast_\infty f \)), see Definitions~\ref{def.fracLp} and~\ref{def.Linfty}.

\medskip

Let one of the following conditions hold true:
\begin{itemize}
	 \item[(a)] $q > n$ and $K$ is a star body;
	 \item[(b)] $q \in (- \infty, n) \setminus \{ 0 \}$ and $K$ is a bounded star body.
\end{itemize} 
Under these assumptions, the complete asymptotic picture of the volume $| \Pi^{\ast, s}_p \, f|$ and dual mixed volume $\vtil_q(K, \Pi^{\ast, s}_p \, f)$ (see Definition~\eqref{vtil}) is summarized in Figure~\ref{fig:two-limit-diagrams}, in which arrows denote the order of limits.
The rigorous statements are provided in Theorems~\ref{thm.volume}--\ref{thm.mixed}.

\begin{figure}[H]
\centering

\begin{minipage}{0.45\textwidth}
\centering
\begin{tikzpicture}[>=latex, scale=1.2]

\node (A) at (0, 2.5) {$|\Pi^{\ast, s}_p f|$};
\node (B) at (4.0, 2.5) {$|\Pi^{\ast}_p f|$};
\node (C) at (0, 0) {$| \Pi^{\ast, s}_\infty f|$};
\node (D) at (4.0, 0) {$|\Pi^\ast_\infty f|$};

\draw[->] (A) -- (B) node[midway, above] {(1)} node[midway, below] {$s \to 1^-$};
\draw[->] (C) -- (D) node[midway, above] {$s \to 1^-$} node[midway, below] {(3)};

\draw[->] (A) -- (C) node[midway, left] {(2)}
                      node[midway, rotate=0, xshift= 4.5ex] {$p \to \infty$};
\draw[->] (B) -- (D) node[midway, right] {(4)}
                      node[midway, rotate=0, xshift=-4.5ex] {$p \to \infty$};

\end{tikzpicture}
\end{minipage}
\hspace{1cm}
\begin{minipage}{0.45\textwidth}
\centering
\begin{tikzpicture}[>=latex, scale=1.2]

\node (A) at (0, 2.5) {$\vtil_q(K, \Pi^{\ast, s}_p f)$};
\node (B) at (4.0, 2.5) {$\vtil_q(K, \Pi^\ast_p \, f)$};
\node (C) at (0, 0) {$\vtil_q(K, \Pi^{\ast, s}_\infty f)$};
\node (D) at (4.0, 0) {$\vtil_q(K, \Pi^\ast_\infty \, f)$};

\draw[->] (A) -- (B) node[midway, above] {(1)} node[midway, below] {$s \to 1^-$};
\draw[->] (C) -- (D) node[midway, above] {$s \to 1^-$} node[midway, below] {(3)};

\draw[->] (A) -- (C) node[midway, left] {(2)}
                      node[midway, rotate=0, xshift=4.5ex] {$p \to \infty$};
\draw[->] (B) -- (D) node[midway, right] {(4)}
                      node[midway, rotate=0, xshift=-4.5ex] {$p \to \infty$};

\end{tikzpicture}
\end{minipage}

\caption{The power-law limit diagrams for volume/dual mixed volume related to $\Pi^{\ast, s}_p \, f$}
\label{fig:two-limit-diagrams}
\end{figure}
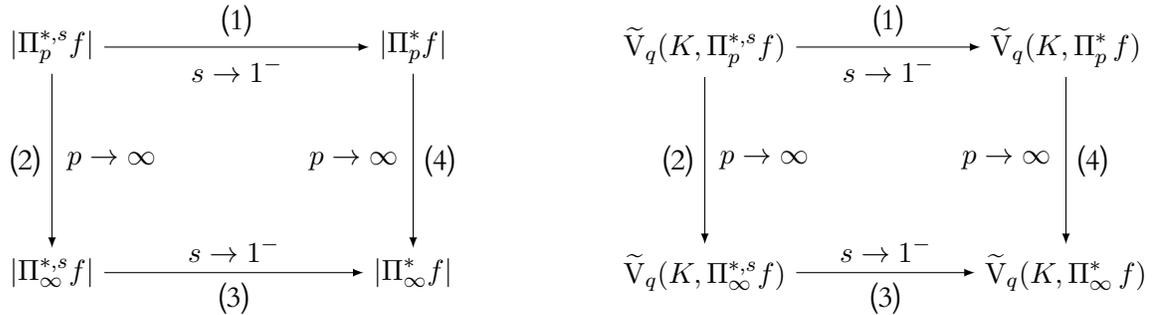

\medskip

For any bounded star body $K$, our second main result describes the asymptotic behavior of
\[
\widetilde{V}_{-sp}(K, \Pi^{\ast,s}_p f)^{1/p}
\]
To this end, we introduce the notion of the $s$--scaled dilation factor $\dila_s(K,L)$ (see Definition~\ref{def.dila}).
Figure~\ref{fig:vtil^1/p} illustrates the complete asymptotic behavior of $\widetilde{V}_{-sp}(K, \Pi^{\ast,s}_p f)^{1/p}$, see Theorem~\ref{thm.mixed-dila} for details.
We further extend the asymptotic analysis for asymmetric variants of $s$--fractional $L^p$ polar projection bodies and as a by-product, we establish the associated geometric inequalities, see~Section~\ref{sec:asymmetric}.

\begin{figure}[H]
\centering
\begin{tikzpicture}[>=latex, scale=1.4]

\node (A) at (0, 2.5) {$\vtil_{-sp}(K, \Pi^{\ast, s}_p f)^{1/p}$};
\node (B) at (4.5, 2.5) {$\vtil_{-p}(K, \Pi^{\ast}_p \, f)^{1/p}$};
\node (C) at (0, 0) {$\dila_s(K, \Pi^{\ast, s}_\infty f)$};
\node (D) at (4.5, 0) {$\dila(K, \Pi^\ast_\infty \, f)$};

\draw[->] (A) -- (B) node[midway, above] {(1)} node[midway, below] {$s \to 1^-$};
\draw[->] (C) -- (D) node[midway, above] {$s \to 1^-$} node[midway, below] {(3)};

\draw[->] (A) -- (C) node[midway, rotate = 0, xshift = 4.5ex] {$p \to \infty$}
                      node[midway, left] {(2)};
\draw[->] (B) -- (D) node[midway, right] {(4)}
                      node[midway, rotate= 0, xshift=-4.5ex] {$p \to \infty$};

\end{tikzpicture}
\caption{The power-law limit diagram for $\vtil_{-sp}(K, \Pi^{\ast, s}_p f)^{1/p}$.}
\label{fig:vtil^1/p}
\end{figure}

\medskip

The remainder of the paper is organized as follows. 
In Section~\ref{sec:prelim}, we review basic facts on star bodies and Sobolev spaces. 
Section~\ref{sec:linfty-bodies} introduces the (symmetric) $s$--fractional $L^\infty$ polar projection bodies and establishes their connection with the $L^p$ counterparts. 
Section~\ref{sec:asymptotics} presents the main results on the asymptotic behavior of volumes and dual mixed volumes, proving the compatibility of the two limiting processes. 
In Section~\ref{sec:asymmetric}, we extend the theory to asymmetric $L^\infty$ polar projection bodies and we derive anisotropic versions of the $L^\infty$ P\'olya--Szeg\H{o} inequality.

\section{Preliminaries}\label{sec:prelim}

In this manuscript, we use $| \cdot |$ to denote the Euclidean norm in $\R^n$. For a measurable set $A \subset \R^n$, we write $|A|$ for its volume (i.e. Lebesgue measure). 
For each $x \in \R^n$ and $r > 0$, we denote by  $B_r(x)$ the ball centered at $x$ with radius $r$.
In the case $x = 0$ and $r = 1$, the unit ball centered at the origin is simply denoted by $B^n := B_1(0)$, and its boundary, the unit sphere, by $\Sphere = \partial B^n$.
In what follows, we collect some basic facts about star bodies and Sobolev spaces.

\medskip

\textbf{Star bodies, dual mixed volumes and Schwarz symmetrization.}

\medskip

We recall the definition of star bodies and related quantities in convex geometry, see \cite{S_2014} for more details.
A set $K \subset \R^n$ is said to be star--shaped with respect to the origin if $\lambda x \in K$ for every $x \in K$ and $\lambda \in [0, 1]$.
Associated with such a set is the gauge function $\| \cdot \|_K : \R^n \to [0, \infty]$, defined as
	\[
		\| x \|_K := \inf \left\{ \lambda > 0: x \in \lambda K  \right\},
	\]
and the radial function $\rho_K : \R^n \setminus \{ 0 \} \to [0, \infty]$ defined as
	\[
		\rho_K(x) := \sup \left\{  \lambda > 0: \lambda x \in K \right\} = \| x \|_K^{-1}.
	\]
With these notations, the volume of $K$ can be computed via the formula
	\begin{equation}\label{gauge-volume}
		| K | = \dfrac{1}{n} \int_{\Sphere} \| \xi \|_K^{-n} d\xi = \dfrac{1}{n} \int_{\Sphere} \rho_K(\xi)^n d\xi.	
	\end{equation}
A star--shaped set $K$ is called a \textit{star body} if its radial function is strictly positive and continuous on $\R^n \setminus \{ 0 \}$.

\medskip

Let $K$ and $L$ be star bodies and let us fix $q \in \R \setminus \{ 0, n \}$.
The \textit{dual mixed volume} of $K$ and $L$ is defined by
	\begin{equation}\label{vtil}
		\vtil_q(K, L) = \dfrac{1}{n} \int_{\Sphere} \rho_K(\xi)^{n - q} \rho_L(\xi)^q d\xi = \dfrac{1}{n} \int_{\Sphere} \| \xi \|_K^{-n + q} \| \xi \|_L^{-q} d\xi.
	\end{equation}
In particular, when $K = L$, we recover the volume $\vtil_q(K, K) = |K|.$
The theory of dual mixed volume has been developed as a dual theory of the classical intrinsic volumes in convex geometry, see~\cite{Lutwak_75, Lutwak_79, Lutwak_88}.

\medskip

Let $E$ be a measurable set in $\R^n$.
The Schwarz symmetral of $E$, denoted by $E^{\star}$, is the centered Euclidean ball with the same volume as $E$.
Recall that for any nonnegative measurable function $f$ on $\R^n$, one can recover its value through the layer cake formula
	\[
		f(x) = \int_0^\infty \pmb{1}_{ \{ f \geq t \}}(x) \, dt, \quad \text{for almost every $x \in \R^n$.
}
	\]
For such functions $f$, its \textit{Schwarz symmetral} $f^\star$ is determined via the formula
	\[
		f^\star(x) := \int_0^\infty \pmb{1}_{ \{ f \geq t \}^\star}(x) \, dt.
	\]
We refer to \cite{B_2009} for more details on~$f^\star$.

\medskip

\textbf{Sobolev spaces.}

\medskip

Let $A \subset \R^n$ be a Borel measurable set.
For \( 1 \leq p < \infty \), we denote by \( L^p(A) \) the space of measurable functions \( f : A \to \mathbb{R} \) whose norm is given by
\[
\|f\|_{L^p(A)} := \left( \int_{A} |f(x)|^p \, dx \right)^{1/p} < \infty,
\]
and by \( L^\infty(A) \) the space of measurable functions with finite essential supremum, equipped with the norm
\[
\|f\|_{L^\infty(A)} := \operatorname*{ess\,sup}_{x \in A} |f(x)|.
\]
We denote the Sobolev space by
	\begin{equation}\label{sobolev}
		W^{1, p}(\R^n) := \left\{ f \in L^p(\R^n): \, | \nabla f | \in L^p(\R^n) \right\}, \quad \text{ for any fixed $p \in [1, \infty]$.}
	\end{equation}
Further, in the case $p \in [1, \infty)$ and $s \in (0, 1)$, the fractional Sobolev space is defined as
	\[
		W^{s, p}(\R^n) := \left\{ f \in L^p(\R^n): \, [f]_{s, p}^p := \iint_{\R^n \times \R^n} \dfrac{|f(x) - f(y)|^p}{|x - y|^{n + sp}} \, dx \, dy < +\infty \right\},
	\]
while  for $p = \infty$, we define
 	\[
		W^{s, \infty} := \left\{ f \in L^\infty(\R^n): \, f \text{ is $s$--H\"older continuous} \right\}.
	\]
It is important to keep in mind the following inclusions:
    \begin{equation}\label{inclu-01}
        W^{1, 1}(\R^n) \cap W^{1, \infty}(\R^n) \subset W^{1, p}(\R^n) \subset W^{s, p}(\R^n), \quad \text{ for every } s \in (0, 1) \text{ and } p \geq 1,
    \end{equation}
see, e.g. \cite[Proposition 2.2]{NPV_2012} for the second.
Furthermore, by Morrey embedding, 
    \begin{equation}\label{inclu-02}
        W^{1, \infty}(\R^n) 
        \subset 
        \mathrm{Lip}(\R^n) \cap L^\infty(\R^n) \subset 
        W^{s, \infty}(\R^n), \quad \text{ for every } s \in (0, 1),
    \end{equation}
where $\mathrm{Lip}(\R^n)$ denotes the space of Lipschitz functions in $\R^n$.

\section{Variants of (\(s\)--fractional) \(L^\infty\) polar projection bodies}\label{sec:linfty-bodies}

In this section, we introduce variants of the ($s$-fractional) $L^\infty$ polar projection bodies, which play a central role in understanding the limits of their $L^p$ counterparts.
We analyze their basic structural properties and establish the convergence of the associated gauge functions as $s \to 1^-$ and $p \to \infty$, which serve as fundamental tools for the asymptotic analysis developed later.

\begin{definition}[$s$--fractional $L^p$ polar projection body]\label{def.fracLp} Fix $p \in [1, \infty)$ and $s \in (0, 1)$.
	\begin{itemize}
		\item[$\bullet$] For any $f \in W^{s, p}(\R^n)$, its $s$--fractional $L^p$ polar projection body $\Pi^{\ast, s}_p \, f$ is a star body defined via the gauge function
		\begin{equation}\label{frac.Lp}
		\| \xi \|^{ps}_{\Pi^{\ast, s}_p f} = p(1 - s)\int_0^\infty t^{-ps - 1} \int_{\R^n} |f(x + t\xi) - f(x)|^p \, dx \, dt, \quad \text{ for every } \xi \in \R^n.
	\end{equation}
	
		\item[$\bullet$] For any $f \in W^{1, p} (\R^n)$, its $L^p$ polar projection body $\Pi^{\ast}_p \, f$ is a star body defined via the gauge function 
		\begin{equation}
		\| \xi \|_{\Pi^*_p \, f}^p = \int_{\R^n} | \langle \nabla f(x), \xi \rangle |^p dx.
	\end{equation}
	\end{itemize}
\end{definition}

\begin{remark}\normalfont
The anisotropic versions of the Gagliardo seminorm and Sobolev norm are obtained by replacing the Euclidean norm \( |\cdot| \) with the gauge function \(\|\cdot\|_{K}\) 
of a convex body \(K\) in the definitions of \([f ]_{s,p}\) and \(\|\nabla f \|_{L^{p}}\), respectively, see, e.g. \cite{L_2014}. 
Moreover, the corresponding Gagliardo seminorm admits a representation in terms 
of the dual mixed volume (cf.~\cite[Page~1098]{HL_2024}):
\begin{equation*}
	\int_{\R^n} \int_{\R^n} 
	\frac{|f(x) - f(y)|^p}{\| x - y \|_K^{\,n + ps}} \, dx \, dy
	= \int_{\Sphere} \rho_K(\xi)^{\,n + ps} \, 
	\rho_{\Pi^{\ast, s}_p f}(\xi)^{-ps} \, d\xi
	= \dfrac{n}{p(1 - s)}\, \widetilde{V}_{-ps}(K, \Pi^{\ast, s}_p f),
\end{equation*}
where \(K\) is a star body.

\end{remark}

\begin{remark}
For notational convenience, we incorporate the factor \(p(1-s)\) into the definition of the gauge function of \(\Pi^{\ast, s}_p f\). 
This choice keeps the formulas compact; without it, the limits as \(s \to 1^{-}\) would otherwise involve more cumbersome expressions. 
In view of \cite[Theorem~9]{HL_2024} and \cite[Theorem~3]{HL_2025}, it follows that, 
for any \(p \in [1, \infty)\),
\begin{align*}
	\displaystyle \lim_{s \to 1^-} \| \xi \|_{\Pi^{\ast, s}_p f} = \| \xi \|_{\Pi^\ast_p f} \quad \text{ and } \quad
	\displaystyle \lim_{s \to 1^-} |\Pi^{\ast, s}_p f| = |\Pi^\ast_p f|.
\end{align*}
Let $K$ be a star body containing the origin in its interior.
Notice that for any $q > n$, the map $\xi \mapsto \| \xi \|_K^{-n + q}$ is bounded in $\Sphere$.
Therefore, we have
\[
	\lim_{s \to 1^-} \widetilde{V}_q\!\left(K, \Pi^{\ast, s}_p f\right) 
	= \widetilde{V}_q\!\left(K, \Pi^\ast_p \, f\right), \text{ for every } q > n, \, p \in [1, + \infty).
\]
Furthermore, if we consider $K$ a bounded star body, the above limit holds for every $q \in (- \infty, n) \setminus \{ 0 \}$ and $p \in [1, + \infty)$.
\end{remark}

\begin{definition}[$s$--fractional $L^\infty$ polar projection bodies]\label{def.Linfty}
Let $s \in (0, 1)$.
\begin{itemize}
    \item[$\bullet$] \
    For any function $f \in W^{s, \infty}(\R^n)$, its \textit{$s$--fractional $L^\infty$ polar projection body} $\Pi^{\ast, s}_\infty  f$ is a star--shaped set defined via the gauge function
	\begin{equation}
		\| \xi \|_{\Pi^{\ast, s}_\infty \, f}^s = \sup_{(x, t) \in \R^n \times (0, + \infty)} \dfrac{1}{t^{s}}|f(x + t\xi) - f(x)|.
	\end{equation}

    \item[$\bullet$]
    For any function $f \in W^{1, \infty}(\R^n)$, its \textit{$L^\infty$ polar projection body} $\Pi^{\ast}_\infty \, f$ is a star--shaped set defined via the gauge function
	\begin{equation}
		\| \xi \|_{\Pi^{\ast}_\infty \, f} = \esssup_{x \in \R^n} |\langle \nabla f(x), \xi \rangle |.
	\end{equation}
\end{itemize}    
\end{definition}

\begin{proposition}\label{prop.exist.infi.polar}
    Let $f \in W^{1, 1}(\R^n) \cap W^{1, \infty}(\R^n)$ be nonzero.
    The following assertions hold true:
    \begin{itemize}
        \item[(i)] 
        For any fixed $s \in (0, 1)$, there exist $c = c(s, f) > 0$ and $\bar p = \bar p(s, f) > 1$ large enough such that $\Pi^{\ast, s}_p f \subset c B^n$ for every $p \geq \bar p$. 
        
        \item[(ii)] 
        For any fixed $s \in (0, 1)$, the set $\Pi^{\ast, s}_\infty \, f$ is origin-symmetric star body containing the origin in its interior.
        Furthermore, there exists $c = c(f) > 0$ and $\bar s = \bar s(f) \in (0, 1)$ such that $\Pi^{\ast, s}_\infty f \subset c B^n$ for every $s \in (\bar s, 1)$.

        \item[(iii)]
        The set $\Pi^{\ast}_\infty \, f$ is bounded origin-symmetric convex body with origin in its interior.
    \end{itemize}
\end{proposition}

\begin{proof}
    \textit{(i)} 
    For any fixed $\alpha \in (0, 1)$, let us denote    
    \[
        \Sset_\alpha := \left\{ x \in \R^n : |f(x)| \geq \alpha \| f \|_{L^\infty(\R^n)} \right\}.
    \]
    
    \begin{claim}\label{claim.W_alpha_bounded}
    		For any fixed $\alpha \in (0, 1)$, the set $\Sset_\alpha$ is bounded.
    \end{claim}
	    
    \medskip
    
    \textit{Proof of Claim \ref{claim.W_alpha_bounded}}
    
	\medskip
	
    Set $L := \mathrm{Lip}(f)$ the Lipschitz constant of $f$ on $\R^n$, which is finite due to $f \in W^{1, \infty}(\R^n)$ and the Morrey embedding.
    Fix $\bar x \in \Sset_\alpha$ and $\bar r = \frac{\alpha \| f \|_{L^\infty(\R^n)}}{2L}$.
   The Lipschitz continuity of $f$ leads to
   	\begin{align*}
   		|f(x)| \geq |f(\bar x)| - L |x - \bar x| \geq \alpha \| f \|_{L^\infty (\R^n)} - L \bar r = \frac{\alpha}{2} \| f \|_{L^\infty(\R^n)}, \text{ for every } x \in B_{\bar r}(\bar x).
   	\end{align*}
	Thus,  $B_{\bar r}(\bar x) \subset \Sset_{\alpha/2}$.
    To prove the boundedness of $\Sset_\alpha$, arguing by contradiction, we assume that there exists a sequence $\{ x_k \} \subset \Sset_\alpha$ satisfying $|x_k| \to \infty$ as $k \to \infty$.
    We can further assume that $|x_k - x_{k'}| > 2\bar r$ and $B_{\bar r}(x_k) \cap B_{\bar r}(x_{k'}) = \varnothing$ for every $k, k' \in \mathbb{N}$.
	Due to the above observation, we note that $B_{\bar r}(x_k) \subset \Sset_{\alpha/2}$ for every $k$.
	It follows that
	\begin{align*}
		\int_{\R^n} |f(x)| dx \geq \sum_{k \in \mathbb N} \int_{B_{\bar r}(x_k)} |f(x)| dx \geq \frac{\alpha}{2} \| f \|_{L^\infty(\R^n)} \sum_{k \in \mathbb N} \underbrace{|B_{\bar r}(x_k)|}_{ \, = \, \bar r^n |B^n|} = \infty,
	\end{align*}
	which contradicts the fact that $f \in L^1(\R^n)$. \hfill $\lozenge$	
	
	\medskip
	
	We continue the proof of \textit{(i)}.
    Thanks to Claim \ref{claim.W_alpha_bounded}, the set $\Sset_{1/4}$ is bounded and so there exists $r_o > 0$ such that $\Sset_{3/4} \subset \Sset_{1/4} \subset r_o B^n$.
	Note that $0 < |\Sset_{3/4}| < + \infty$ and for any $p \in [1, + \infty)$, we have 
	\[
	\| f \|_{L^p(\Sset_{3/4})} \geq \frac{3}{4} \| f \|_{L^\infty(\R^n)} |\Sset_{3/4}|^{1/p}.
	\]
	Fix $\xi \in \Sphere$.
    For any $t > 2r_o$, the sets $r_oB^n$ and $\Sset_{3/4} - t\xi$ are disjoint and hence we get 
    \[
    	\| f \|_{L^p(\Sset_{3/4} - t \xi)} \leq \dfrac{1}{4} \| f \|_{L^\infty(\R^n)} |\Sset_{3/4} - t\xi|^{1/p} = \dfrac{1}{4} \| f \|_{L^\infty(\R^n)} |\Sset_{3/4}|^{1/p}.
    \]
    Using the triangle inequality, for any $t > 2r_o$, we have
    \begin{equation}\label{differ-Lp}
    	\begin{split}
    	\| f(\cdot + t \xi) - f \|_{L^p(\R^n)} 
    	\geq &  ~ 
    	~ \underbrace{\| f(\cdot + t \xi) - f \|_{L^p(\Sset_{3/4} - t\xi)}}_{= \, \| f   - f(~\cdot ~- t \xi) \|_{L^p(\Sset_{3/4})}} \\
    	\geq & ~ 
    	~ \| f \|_{L^p(\Sset_{3/4})} - \| f (\cdot - t \xi) \|_{L^p(\Sset_{3/4})} \\
    	\geq & ~
	\frac{3}{4} \| f \|_{L^\infty(\R^n)} |\Sset_{3/4}|^{1/p} - \dfrac{1}{4} \| f \|_{L^\infty(\R^n)} |\Sset_{3/4}|^{1/p} \\
    = & ~  \dfrac{1}{2}\| f \|_{L^\infty(\R^n)} |\Sset_{3/4}|^{1/p} > 0.
    	\end{split}
    \end{equation}
    Consequently, we obtain
    \begin{align*}
    	 \| \xi \|^{ps}_{\Pi^{\ast, s}_p \, f} 
         = & ~  p(1 - s) \int_0^\infty t^{-ps - 1} \| f(\cdot + t\xi) - f \|^p_{L^p(\R^n)} \, dt \\
    	 \geq & 
    	 ~ p(1 - s) \left( \dfrac{\| f \|_{L^\infty(\R^n)} |\Sset_{3/4}|^{1/p}}{2}  \right)^p \int_{2r_o}^\infty t^{-ps - 1}dt \\
    	 = & 
    	 ~ \dfrac{(1 - s)(2r_o)^{-ps}}{s}  \left( \dfrac{\| f \|_{L^\infty(\R^n)} |\Sset_{3/4}|^{1/p}}{2}  \right)^p,
    \end{align*}
   	and thus
   	\begin{align*}
   		\| \xi \|^{s}_{\Pi^{\ast, s}_p \, f} \geq \dfrac{1}{2} (2r_o)^{-s}\| f \|_{L^\infty(\R^n)} |\Sset_{3/4}|^{1/p} \left( \dfrac{1 - s}{s} \right)^{1/p}.
   	\end{align*}
 Since $\textstyle\lim_{p \to \infty} |\Sset_{3/4}|^{1/p}(s^{-1} - 1)^{1/p} = 1$, there exist $\bar p  = \bar p(s, f) > 1$ and $c = c(s, f) > 0$ such that $\| \xi \|_{\Pi^{\ast, s}_p \, f} \geq c$ for every $p \geq \bar p$. 
 Therefore, the conclusion follows.
   
    \medskip
    
    \textit{(ii)} 
    Notice that 
    \[
        \| - \xi \|_{\Pi^{\ast, s}_\infty \, f} = \sup_{(x, t)} \dfrac{1}{t^s} |f(x - t\xi) - f(x)| = \sup_{(\tilde x, t)} \dfrac{1}{t^s} |f(\tilde x) - f(\tilde x + t\xi)| = \| \xi \|_{\Pi^{\ast, s}_\infty \, f},
    \]
    and hence $\Pi^{\ast, s}_\infty \, f$ is origin-symmetric.
    
    \medskip

    Let us check that the origin lies in the interior of $\Pi^{\ast, s}_\infty \, f$.
    One can directly see that $\xi \mapsto \| \xi \|_{\Pi^{\ast, s}_\infty \, f}^s$ satisfies the triangle inequality.
    Indeed, for each $\xi, \eta \in \R^n$, we have
    \begin{align*}
        \| \xi + \eta \|_{\Pi^{\ast, s}_\infty \, f}^s 
        = & ~ 
        \sup_{t > 0}
        \dfrac{1}{t^s} \| f(\cdot + t\xi + t\eta) - f \|_{L^\infty(\R^n)} \\
        \leq & ~
        \sup_{t > 0} 
        \dfrac{1}{t^s} \Big(
        \underbrace{\| f(\cdot + t\xi + t\eta) - f(\cdot + t \eta) \|_{L^\infty(\R^n)}}_{ = \, \| f(~\cdot~ + t\xi) - f \|_{L^\infty(\R^n)}} 
        + 
        \| f(\cdot + t\eta) - f \|_{L^\infty(\R^n)} 
        \Big) 
        \\
        \leq & ~ 
        \sup_{t > 0} \frac{1}{t^s} \| f(\cdot + t\xi) - f \|_{L^\infty(\R^n)} + \sup_{t > 0} \dfrac{1}{t^s}\| f(\cdot + t\eta) - f \|_{L^\infty(\R^n)} \\
        = & ~
        \| \xi \|^s_{\Pi^{\ast, s}_\infty \, f} 
        + \| \eta \|^s_{\Pi^{\ast, s}_\infty \, f} \,.
    \end{align*}
 	Notice that $\| t \xi \|_{\Pi^{\ast, s}_\infty \, f} = |t| \| \mathrm{sign}(t)\xi \|_{\Pi^{\ast, s}_\infty \, f}$ for every $t \in \R$ and $\xi \in \R^n$.
 	For any $x \in \R^n$, we write $x = \textstyle\sum_j x_j e_j$, where $\{ e_j \}_{j = 1}^d$ is the canonical basis of $\R^n$.
 	By the triangle inequality above, we get
 	\begin{equation}\label{origin-interior}
 		\| x \|_{\Pi^{\ast, s}_\infty f} \leq \left( \sum_j \| x_j e_j \|_{\Pi^{\ast, s}_\infty f}^s \right)^{1/s} = \left( \sum_j |x_j| \| \mathrm{sign}(x_j) e_j \|^s_{\Pi^{\ast, s}_\infty f}  \right)^{1/s} \leq a |x|,
 	\end{equation}
where $a > 0$ is independent of $x$. 
	This implies that $\Pi^{\ast, s}_\infty \, f$ contains the origin in its interior.

    \medskip
    
	To conclude that $\Pi^{\ast, s}_\infty \, f$ is a star body, it remains to prove that the gauge function $\| \cdot \|_{\Pi^{\ast, s}_\infty \, f}$ is continuous.
	Let $\{ \xi_k \}$ be a sequence converging to some $\xi \in \R^n$. 
 	Since the map $\xi \mapsto \| \xi \|_{\Pi^{\ast, s}_\infty \, f}^s$ satisfies the triangle inequality, we have
    \begin{align*}
        \Big\vert \| \xi_k \|_{\Pi^{\ast, s}_\infty \, f}^s - \| \xi \|_{\Pi^{\ast, s}_\infty \, f}^s \Big\vert
        \leq 
        \| \xi_k - \xi \|_{\Pi^{\ast, s}_\infty \, f}^s \underbrace{\, \leq \,}_{\eqref{origin-interior}} a^s | \xi_k - \xi|^s \to 0 \text{ as } k \to \infty.
    \end{align*}

    \medskip

    It remains to prove the boundedness of $\Pi^{\ast, s}_\infty \, f$ as $s$ is sufficiently close to $1$, that is, there exists $c = c(f) > 0$ such that $\| \xi \|_{\Pi^{\ast, s}_\infty \, f} \geq c$ for every $s$ closed to $1$.
    Let $r > 1$ be such that $\| f \|_{L^1(r B^n)} \geq \textstyle\frac{2}{3} \| f \|_{L^1(\R^n)}$.
    Fix $\xi \in \Sphere$.
    For any $t > 2r$, the two sets $rB^n$ and $rB^n - t\xi$ are disjoint and we also have $\| f \|_{L^1(rB^n - t\xi)} \leq \textstyle\frac{1}{3} \| f \|_{L^1(\R^n)}$.
    With these remarks, using the triangle inequality, we obtain the following estimate for any $t > 2r$,
    \begin{align*}
    	\| f(\cdot + t\xi) -  f\|_{L^\infty(\R^n)}
    	\geq & 
    	~ \| f(\cdot + t\xi) -  f\|_{L^\infty(rB^n - t\xi)} \\
    	\geq & 
    	~ \dfrac{1}{|rB^n - t\xi|} \underbrace{\| f(\cdot + t\xi)  - f \|_{L^1(rB^n - t\xi)} }_{\, = \, \| f - f(~\cdot~ - t \xi)\|_{L^1(rB^n)}} \\
    	\geq & 
    	~ \dfrac{1}{r^n |B^n|} \Big( \| f \|_{L^1(rB^n)} - \underbrace{\| f(\cdot - t\xi) \|_{L^1(rB^n)}}_{ \, = \, \| f \|_{L^1(rB^n - t\xi)} } \Big) 
     \\
    	\geq & 
    	~ \dfrac{1}{r^n |B^n|} \left( \dfrac{2}{3} \| f \|_{L^1(\R^n)} - \dfrac{1}{3} \| f \|_{L^1(\R^n)} \right) =  \dfrac{1}{3 r^n |B^n|} \| f \|_{L^1(\R^n)}.
    \end{align*}
    It follows from the above estimate and the definition of $\| \cdot \|_{\Pi^{\ast, s}_\infty \, f}$  that
	\begin{align*}
		\| \xi \|_{\Pi^{\ast, s}_\infty \, f} = \left( \sup_{t > 0} \dfrac{1}{t^s} \| f(\cdot + t\xi) - f \|_{L^\infty(\R^n)} \right)^{1/s} \geq \dfrac{\upsilon^{1/s}}{2r} 
		\quad \text{ with } \quad
		\upsilon = \dfrac{\| f \|_{L^1(\R^n)}}{3r^n |B^n|}.
	\end{align*}
	Note that $\textstyle\lim_{s \to 1^-} \upsilon^{1/s} = \upsilon$ and hence there exists $\bar s = \bar s(f) \in (0, 1)$ such that $\upsilon^{1/s} \geq \upsilon/2$ for all $s \in (\bar s, 1)$.
	This concludes the boundedness of $\Pi^{\ast, s}_\infty \, f$.
	
 	\medskip
 	
 	\textit{(iii)}
 	 The proof is direct.
 	 Indeed, we have, for every $\xi, \eta \in \R^n$, that
 	 \begin{align*}
 	 	\| \xi + \eta \|_{\Pi^\ast_\infty \, f} = & ~
        \esssup_{x \in \R^n} | \langle \nabla f(x), \xi + \eta \rangle| \\
        \leq & ~ \esssup_{x \in \R^n} | \langle \nabla f(x), \xi \rangle | + \esssup_{x \in \R^n} |\langle \nabla f(x), \eta \rangle| = \| \xi \|_{\Pi^\ast_\infty \, f} + \| \eta \|_{\Pi^\ast_\infty \, f}
 	 \end{align*}
 	 and
 	 \begin{align*}
 	 	\| \xi \|_{\Pi^\ast_\infty \, f } \leq \| \nabla f \|_{L^\infty(\R^n)} |\xi|.
 	 \end{align*}
 	 As a consequence of the above observations, we infer that $\xi \mapsto \| \xi \|_{\Pi^\ast_\infty \, f}$ is continuous.

        \medskip

        It remains to prove that $\Pi^\ast_\infty \, f$ is bounded, that is, $\textstyle\inf_{\xi \in \Sphere} \| \xi \|_{\Pi^\ast_\infty \, f} > 0$.
        Arguing by contradiction, we assume that $\textstyle\inf_{\xi \in \Sphere} \| \xi \|_{\Pi^\ast_\infty \, f} = 0$.
        Consequently, there exists $\xi_o \in \Sphere$ such that 
            \begin{align*}
                \esssup_{x \in \R^n} | \langle \nabla f(x), \xi_o \rangle | = 0.
            \end{align*}
Consequently, for a.e $x \in \R^n$, $f$ is constant along the line $x + t\xi_o$.
        Since $f \in L^1(\R^n)$, we infer that $f \equiv 0$, which makes a contradiction.
 	The proof of Proposition \ref{prop.exist.infi.polar} is complete.
\end{proof}

\begin{proposition}\label{prop.lim.gauge}
    Let $f \in W^{1, 1}(\R^n) \cap W^{1, \infty}(\R^n)$ be nonzero.
    Then, the following limits hold:
    \begin{itemize}
        \item[(i)] for any fixed $s \in (0, 1)$, one has $\| \xi \|_{\Pi^{\ast, s}_p f} \xlongrightarrow{\, p \to + \infty \,} \| \xi \|_{\Pi^{\ast, s}_\infty \, f}$ a.e. $\xi \in \Sphere$;

        \item[(ii)] $\| \xi \|_{\Pi^\ast_p \, f} \xlongrightarrow{\, p \to  \infty \,} \| \xi \|_{\Pi^\ast_\infty \, f}$ for every $\xi \in \Sphere$;

        \item[(iii)] $\| \xi \|_{\Pi^{\ast, s}_\infty \, f} \xlongrightarrow{\, s \to 1^- \,} \| \xi \|_{\Pi^\ast_\infty \, f}$ for every $\xi \in \Sphere$.
    \end{itemize}
\end{proposition}


\begin{proof}
\textit{(i)} 
Let us fix $s \in (0, 1)$.
From the inclusion~\eqref{inclu-01}, observe that $f \in W^{1, s}(\R^n)$.
Recall that the Gagliardo seminorm of $f$ can be written after a polar change of variable as follows: 
\begin{align*}
	\int_{\R^n} \int_{\R^n} \dfrac{|f(x) - f(y)|}{|x - y|^{n + s}} \, dx \, dy = \int_{\Sphere} \int_0^\infty \int_{\R^n} \dfrac{|f(x + t\xi) - f(x)|}{t^{s + 1}} \, dx \, dt \, d\xi.
\end{align*}
Therefore, for a.e $\xi \in \Sphere$, one has
	\begin{equation}\label{int-s+1-fini}
		\int_0^\infty \int_{\R^n} \dfrac{|f(x + t \xi) - f(x)|}{t^{s + 1}} \, dx \, dt < + \infty.
	\end{equation}

\smallskip

On the one hand, for any fixed $p > 1$, a direct computation yields
\begin{align*}
    & ~ \int_0^\infty t^{-sp - 1} \int_{\R^n} |f(x + t \xi) - f(x)|^p \,dx\,dt \\
    = & ~ \int_0^\infty \int_{\R^n} \dfrac{|f(x + t\xi) - f(x)|^{p - 1}}{t^{s(p - 1)}} \dfrac{|f(x + t\xi) - f(x)|}{t^{s + 1}} \, dx\,dt  \\
    \leq & ~ \| \xi \|^{s(p - 1)}_{\Pi^{\ast, s}_\infty f} \int_0^\infty \int_{\R^n}  \dfrac{|f(x + t\xi) - f(x)|}{t^{s + 1}} \, dx \, dt.
\end{align*}
This implies that 
    \begin{align*}
        \| \xi \|^s_{\Pi^{\ast, s}_p \, f} \leq  \| \xi \|^{s \left(1 - 1/p\right)}_{\Pi^{\ast, s}_\infty \, f}
        (p(1 - s))^{1/p}
        \left( \int_0^\infty \int_{\R^n}  \dfrac{|f(x + t\xi) - f(x)|}{t^{s + 1}} \, dx \, dt\right)^{1/p}.
    \end{align*}
Note that $\textstyle\lim_{p \to \infty} (p(1 - s))^{1/p} = 1$.
Thanks to \eqref{int-s+1-fini}, we immediately get
    \begin{equation}\label{limsup-01}
        \limsup_{p \to \infty} \| \xi \|_{\Pi^{\ast, s}_p \, f} \leq \| \xi \|_{\Pi^{\ast, s}_\infty \, f} \,.
    \end{equation}

\medskip

On the other hand, for a fixed $\varepsilon > 0$ small enough, set
    \begin{equation}\label{A_s_epsi}
        A_{s, \varepsilon} := \left \{   (x, t) \in \R^n \times (0, + \infty): \dfrac{|f(x + t\xi) - f(x)|}{t^s} > \|\xi \|^s_{\Pi^{\ast, s}_\infty f} - \varepsilon \right\}.
    \end{equation}
Notice that $|A_{s, \varepsilon}| > 0$.
Then, we compute, for any fixed $p > 1$,
    \begin{equation}\label{est.A_s_epsi}
    	\begin{split}
         \frac{1}{p(1- s)}\|\xi\|^{ps}_{\Pi^{\ast, s}_p \, f} = & ~ \int_0^\infty t^{-sp - 1} \int_{\R^n} |f(x + t \xi) - f(x)|^p \, dx \, dt \\
         \geq &  ~  \iint_{A_{s, \varepsilon}} \dfrac{|f(x + t\xi) - f(x)|^p}{t^{sp + 1}} \, dx \, dt \\
         \geq & ~ \left( \| \xi \|^s_{\Pi^{\ast, s}_\infty \, f} - \varepsilon \right)^{p - 1} \iint_{A_{s,\varepsilon}} \dfrac{t^{s(p - 1)}|f(x +t\xi) - f(x)|}{t^{sp + 1}} \, dx \, dt  \\
         = & ~ \left( \| \xi \|^s_{\Pi^{\ast, s}_\infty \, f} - \varepsilon \right)^{p - 1}\iint_{A_{s, \varepsilon}} \dfrac{|f(x +t\xi) - f(x)|}{t^{s + 1}} \, dx \, dt.
         \end{split}
    \end{equation}
Therefore, thanks to the fact that $\textstyle\lim_{p \to \infty} (p(1-s))^{1/p} = 1$, we get
    \[
        \liminf_{p \to \infty} \| \xi \|^s_{\Pi^{\ast, s}_p f} \geq \| \xi \|^s_{\Pi^{\ast, s}_\infty f} - \varepsilon.
    \]
Since $\varepsilon > 0$ is arbitrary, we obtain
    \begin{equation}\label{liminf-01}
        \liminf_{p \to \infty} \| \xi \|_{\Pi^{\ast, s}_p f} \geq \| \xi \|_{\Pi^{\ast, s}_\infty f}.
    \end{equation}
Combining \eqref{limsup-01} and \eqref{liminf-01}, we conclude that $\textstyle\lim_{p \to \infty}  \| \xi \|_{\Pi^{\ast, s}_p f} = \| \xi \|_{\Pi^{\ast, s}_\infty  f}$.

\medskip

\textit{(ii)}
Notice that $\| \xi \|_{\Pi^\ast_p \, f} = \| \langle \nabla f (\cdot), \xi \rangle \|_{L^p(\R^n)}$ and $\| \xi \|_{\Pi^\ast_\infty \, f} = \| \langle \nabla f(\cdot), \xi \rangle \|_{L^\infty(\R^n)}$.
Consequently, using the fact that $f \in W^{1, 1}(\R^n) \cap W^{1, \infty}(\R^n)$, it is fundamental to check that $\textstyle\lim_{p \to \infty} \| \xi \|_{\Pi^\ast_p \, f} = \| \xi \|_{\Pi^\ast_\infty \, f}$. 

\medskip

\textit{(iii)} 
Fix $\xi \in \Sphere$ and $\delta > 0$ small enough such that
\[
    \delta < \dfrac{\| \xi \|_{\Pi^\ast_\infty \, f} }{1 +  \| \xi \|_{\Pi^\ast_\infty\, f} }.
\]
By definitions, there exist $x_o \in \R^n$ and $t_o > 0$
such that 
    \begin{equation}\label{NUM-01}
    \begin{split}
        (1 - \delta) \| \xi \|_{\Pi^\ast_\infty \, f} 
        \leq & ~
        | \langle \nabla f(x_o), \xi \rangle | \\
        = & ~ \lim_{t \to 0^+} \dfrac{1}{t} |f(x_o + t\xi) - f(x_o)| \leq \dfrac{1}{t_o} |f(x_o + t_o \xi) - f(x_o)| + \delta.
    \end{split}
    \end{equation}
Furthermore, we have
    \begin{equation}\label{NUM-02}
        \dfrac{1}{t_o} |f(x_o + t_o \xi) - f(x_o)| 
        = 
        t_o^{s - 1} \dfrac{|f(x_o + t_o \xi) - f(x_o)|}{t_o^s}
        \leq t_o^{s - 1} \| \xi \|^s_{\Pi^{\ast, s}_\infty \, f} \quad \text{ for every } s \in (0, 1).
    \end{equation}
It follows from the estimates \eqref{NUM-01} and  \eqref{NUM-02} that
    \begin{align*}
        \left( 
        (1 - \delta) \| \xi \|_{\Pi^\ast_\infty \, f} - \delta
        \right)^{1/s} \leq t_o^{1 - 1/s} \| \xi \|_{\Pi^{\ast, s}_\infty \, f} \quad \text{ for every } s \in (0, 1).
    \end{align*}
Letting first $s \to 1^-$ and then $\delta \searrow 0$, we deduce
    \begin{equation}\label{liminf-sup-t^s}
        \| \xi \|_{\Pi^\ast_\infty \, f} \leq \liminf_{s \to 1^-}  \| \xi \|_{\Pi^{\ast, s}_\infty \, f} \, .
    \end{equation}
    
\medskip

It remains to check the limsup inequality: $\textstyle\limsup_{s \to 1^-} \| \xi \|_{\Pi^{\ast, s}_\infty \, f} \leq \| \xi \|_{\Pi^\ast_\infty \, f}$.
Fix $s \in  (0, 1)$ and $\varepsilon = \varepsilon(s, \xi) > 0$ such that $\varepsilon < \| \xi \|_{\Pi^{\ast, s}_\infty f}^s$.
Let $A_{s, \varepsilon}$ be defined as in \eqref{A_s_epsi}.
For any $p > 1$, we recall from \eqref{est.A_s_epsi} that
	\begin{equation}\label{esti.iii-01}
        \begin{split}
		\dfrac{1}{p(1 - s)}\| \xi \|^{sp}_{\Pi^{\ast, s}_p \, f} = & ~ \int_0^\infty t^{-sp - 1} \int_{\R^n} |f(x + t \xi) - f(x)|^p \, dx \, dt \\
        \geq & ~  \left( \| \xi \|^s_{\Pi^{\ast, s}_\infty \, f} - \varepsilon \right)^{p - 1}\iint_{A_{s, \varepsilon}} \dfrac{|f(x +t\xi) - f(x)|}{t^{s + 1}} \, dx \, dt
        \end{split}
	\end{equation}
Due to the Lipschitz continuity of $f$ and the fact that Sobolev maps are absolutely continuous along almost every line, it holds $|f(x + t \xi) - f(x)| \leq t \| \xi \|_{\Pi^\ast_\infty \, f}$ for everywhere $(x, t) \in \R^n \times (0, + \infty)$.
With a direct computation, we get
	\begin{equation}\label{Mdiff.01}
        \begin{split}
		& ~ \dfrac{1}{p(1 - s)}\| \xi \|^{sp}_{\Pi^{\ast, s}_p \, f} \\
        = &~ \int_0^\infty t^{-sp - 1} \int_{\R^n} |f(x + t \xi) - f(x)|^{s(p - 1)} |f(x + t \xi) - f(x)|^{p(1 - s) + s} \, dx \, dt \\
		\leq & ~ \| \xi \|_{\Pi^\ast_\infty \, f}^{s(p - 1)} \int_0^\infty \int_{\R^n} \dfrac{|f(x + t \xi) - f(x)|^{p(1 - s) + s}}{t^{s + 1}} \, dx \, dt.
        \end{split}
	\end{equation}
We continue by splitting the above integral into two parts.
Fix $\varrho > 0$ and set $p_s := p(1 - s) + s$.
Using H\"older's inequality with the exponents $p_s$ and $p_s/(p_s - 1)$ and Fubini's theorem, we obtain
	\begin{equation}\label{Mdiff.02}
        \begin{alignedat}{3}
		~ \int_0^\varrho \int_{\R^n} \dfrac{|f(x + t \xi) - f(x)|^{p_s}}{t^{s + 1}} \, dx \, dt  \leq &  ~ 
        \int_0^\varrho \int_{\R^n} t^{-(s + 1)}  \Big( \int_0^t | \langle \nabla f(x + \tau \xi), \xi \rangle | \, d\tau \Big)^{p_s} \, dx \, dt &
        \\
        \leq & ~ 
        \int_0^{\varrho} t^{p_s - s - 2} \int_{\R^n} \int_0^t |\langle \nabla f(x + \tau \xi), \xi \rangle |^{p_s} \, d\tau \, dx \, dt & 
        \\
        \leq & ~ \int_0^{\varrho} t^{p_s - s - 2} \int_0^t \underbrace{\int_{\R^n} |\langle \nabla f(x + \tau \xi), \xi \rangle |^{p_s} \, dx}_{\, = \, \| \langle \nabla f (\cdot), \xi \rangle \|_{L^{p_s}(\R^n)}^{p_s}} \, d\tau \, dt &
        \\
        \leq & ~ \| \nabla f \|_{L^{p_s}(\R^n)}^{p_s} \int_0^\varrho t^{p_s - s - 1} dt & 
        \\
		\leq & ~ \dfrac{\varrho^{p(1 - s)}}{p(1 - s)} \| \nabla f \|^{(p - 1)(1 - s)}_{L^\infty(\R^n)} \| \nabla f \|_{L^1(\R^n)}, &
        \end{alignedat}
	\end{equation}
where we have used the fact that $p_s - s - 1 = p(1 - s) - 1$ and a simple interpolation inequality
\[
\| \nabla f \|_{L^{p_s}(\R^n)}^{p_s} \leq \| \nabla f \|^{(p - 1)(1 - s)}_{L^\infty(\R^n)} \| \nabla f \|_{L^1(\R^n)}.
\]
Analogously, using the fact that $f \in L^1(\R^n) \cap L^\infty(\R^n)$, we get an estimate in the domain $(\varrho, \infty) \times \R^n$ as follows
	\begin{equation}\label{Mdiff.03}
    \begin{split}
		\int_\varrho^\infty \int_{\R^n} \dfrac{|f(x + t \xi) - f(x)|^{p_s}}{t^{s + 1}} \,dx \,dt 
		\leq & ~  \| 2f \|^{p_s}_{L^{p_s}(\R^n)} \int_\varrho^\infty \dfrac{1}{t^{s + 1}} dt \\
		\leq & ~ \dfrac{1}{s\varrho^s} \| 2f \|^{(p - 1)(1 - s)}_{L^\infty(\R^n)} \| 2f \|_{L^1(\R^n)}.
    \end{split}
	\end{equation}
Combining~\eqref{Mdiff.01}, \eqref{Mdiff.02} and \eqref{Mdiff.03}, we obtain the following upper bound
	\begin{equation}\label{esti.iii-02}
    \begin{split}
		& ~ \dfrac{1}{p(1 - s)} \| \xi \|^{sp}_{\Pi^{\ast, s}_p \, f} 
        \\
		\leq & ~ \| 
		\xi \|_{\Pi^\ast_\infty \, f}^{s(p - 1)} 
		\left(
		\dfrac{\varrho^{p(1 - s)}}{p(1 - s)} \| \nabla f \|^{(p - 1)(1 - s)}_{L^\infty(\R^n)} \| \nabla f \|_{L^1(\R^n)} 
		+
		\dfrac{1}{s \varrho^s} \| 2f \|^{(p - 1)(1 - s)}_{L^\infty(\R^n)} \| 2f \|_{L^1(\R^n)} 
		\right).
    \end{split}
	\end{equation}
Observe that for any fixed $p > 2$, one has $(r + s)^{\frac{1}{p - 1}} \leq r^{\frac{1}{p - 1}} + s^{\frac{1}{p - 1}}$ for every $r, s \geq 0$.
Therefore, combining the estimates \eqref{esti.iii-01} and \eqref{esti.iii-02} and then taking the $(p - 1)$th root, we get
	\begin{equation}
	\begin{split}
		& \| \xi \|_{\Pi^\ast_\infty \, f}^s 
  		\left(
     	\frac{\varrho^{(1 - s)\frac{p}{p - 1}}}{(p(1 - s))^{\frac{1}{p - 1}}}
     	\| \nabla f \|^{1 - s}_{L^\infty(\mathbb{R}^d)}
     	\| \nabla f \|_{L^1(\mathbb{R}^n)}^{\frac{1}{p - 1}}
     	+ 
     	\frac{1}{(s\varrho^s)^{\frac{1}{p - 1}}}
     	\| 2f \|^{1 - s}_{L^\infty(\mathbb{R}^n)}
     	\| 2f \|_{L^1(\mathbb{R}^d)}^{\frac{1}{p - 1}}
  		\right)
		\\
		\geq & ~
		\left( \| \xi \|^s_{\Pi^{\ast, s}_\infty \, f} - \varepsilon 			\right)
  		\left( \iint_{A_{s, \varepsilon}} 
  		\frac{|f(x + t\xi) - f(x)|}{t^{s + 1}} \, dx \, dt \right)^{\frac{1}{p - 1}}.
	\end{split}
	\end{equation}
Since $\textstyle\lim_{p \to \infty} (p(1 - s))^{\frac{1}{p - 1}} = 1$, letting first $p \to \infty$ and then taking the $s$th root in the above inequality, we arrive at
	\begin{align*}
		\| \xi \|_{\Pi^\ast_\infty \, f} \left( 
		\varrho^{1 - s} \| \nabla f \|^{1 - s}_{L^\infty(\mathbb{R}^n)} 
		+
		\| 2f \|^{1 - s}_{L^\infty(\mathbb{R}^d)}
		\right)^{1/s}
		\geq  \left( \| \xi \|^s_{\Pi^{\ast, s}_\infty \, f} - \varepsilon \right)^{1/s}, 
	\end{align*}
for every $\varrho \in (0, 1) \text{ and } \varepsilon (0, \| \xi \|^s_{\Pi^{\ast, s}_\infty \, f})$.
Lastly, letting $\varrho \searrow 0$, $\varepsilon \searrow 0$ and then $s \to 1^-$ in the above inequality, we infer that
	\begin{equation}\label{limsup-sup-t^s}
		\limsup_{s \to 1^-} \| \xi \|_{\Pi^{\ast, s}_\infty \, f} \leq \| \xi \|_{\Pi^\ast_\infty \, f} \lim_{s \to 1^-} \| 2 f \|^{(1 - s)/s}_{L^\infty(\R^n)} =  \| \xi \|_{\Pi^\ast_\infty \, f}.
	\end{equation}
Combining \eqref{liminf-sup-t^s} and \eqref{limsup-sup-t^s}, we can finally conclude that $\textstyle\lim_{s \to 1^-} \| \xi \|_{\Pi^{\ast, s}_\infty \, f} = \| \xi \|_{\Pi^\ast_\infty \, f}$.
This completes the proof of Proposition \ref{prop.lim.gauge}.
\end{proof}

\begin{remark}\label{rem.bounded-gauge}\normalfont
$ \, $

\smallskip

\textit{(i)}
A careful inspection of the proof of Proposition~\ref{prop.lim.gauge} shows a uniform bound in~$p$:
for any fixed ${s \in (0,1)}$, there exists a constant $M > 0$ (independent of~$p$) such that for sufficiently large~$p$,
\begin{equation}\label{ps_upperbound}
    \| \xi \|_{\Pi^{\ast, s}_p f} \le M 
    \quad \text{for a.e. } \xi \in \mathbb{S}^{n-1}.
\end{equation}

Indeed, applying the estimate \eqref{esti.iii-02} to the case $\varrho = 1$ and taking into account that $\| \xi \|_{\Pi^\ast_\infty \, f} \leq \| \nabla f \|_{L^\infty(\R^n)}$ we have 
\begin{align*}
	\| \xi \|^{s}_{\Pi^{\ast, s}_p \, f} \leq \| \nabla f \|^{s(1- 1/p)}_{L^\infty(\R^n)} 
	\Big( &
	\| \nabla f  \|_{L^\infty(\R^n)}^{(1 - 1/p)(1 - s)} \| \nabla f \|_{L^1(\R^n)}^{1/p}
    \\
    &
	+ 	s^{-1} (p(1 - s))^{1/p} \| 2f \|_{L^\infty(\R^n)}^{(1 - 1/p)(1 - s)} \| 2f \|_{L^1(\R^n)}^{1/p}
	\Big).
\end{align*}
Since $\textstyle\lim_{p \to \infty} (p(1 - s))^{1/p} = 1$ and $\textstyle\lim_{p \to \infty} \alpha^{1/p} = 1$ for $\alpha > 0$, the above estimate implies our desired conclusion~\eqref{ps_upperbound}.

\medskip

\textit{(ii)} Notice that by Morrey embedding, any function $f \in W^{1, \infty}(\R^n)$ is Lipschitz, where the Lipschitz constant depends only on dimension and $\| f \|_{W^{1, \infty}(\R^n)}$.
Therefore, there exists a constant $M$ depending only on dimension $n$, $\| f \|_{L^\infty}$ and $\| \nabla  f \|_{L^\infty}$ such that
\begin{align*}
    |f(x) - f(y)| \leq M |x - y|^s \text{ for every }x, y \in \R^n \text{ and } s \in (0, 1).
\end{align*}
In particular, we have
\begin{align*}
    \| \xi \|_{\Pi^{\ast, s}_\infty \, f} \leq 1 + M^2 \quad \text{ for every } \xi \in \Sphere \text{ and } s \in (1/2, 1).
\end{align*}

\medskip

\textit{(iii)} A simple calculation shows that for any nonzero $f \in W^{1, 1}(\R^n) \cap W^{1, \infty}(\R^n)$, the following estimate holds
\begin{align*}
    \| \xi \|_{\Pi^{\ast}_p \, f} \leq \| \nabla f \|_{L^\infty(\R^n)} + \| \nabla f \|_{L^1(\R^n)} \quad \text{ for every } \xi \in \Sphere \text{ and } p > 1.
\end{align*}
\end{remark}

\section{Asymptotic behavior}\label{sec:asymptotics}

This section presents our main results on the limiting behavior of volumes and dual mixed volumes associated with the $s$--fractional $L^p$ polar projection bodies, see Theorems \ref{thm.volume}--\ref{thm.mixed}. 
The second main result, Theorem~\ref{thm.mixed-dila}, captures asymptotics of the quantity $\vtil_{-sp}(K, \Pi^{\ast, s}_p \, f)^{1/p}$.

\begin{theorem}\label{thm.volume}
    Let $f \in W^{1, 1}(\R^n) \cap W^{1, \infty}(\R^n)$ be nonzero.
    Then, the following assertions hold true:
    \begin{itemize}
        \item[(i)] for any fixed $s \in (0, 1)$, one has $        \textstyle\lim_{p \to \infty} |\Pi^{\ast, s}_p  f| =  |\Pi^{\ast, s}_\infty f|$; 

        \item[(ii)] $\textstyle\lim_{p \to \infty} |\Pi^\ast_p  f| = |\Pi^\ast_\infty  f|$;

        \item[(iii)] $\textstyle\lim_{s \to 1^-} |\Pi^{\ast, s}_\infty \, f| = |\Pi^\ast_\infty \, f|$.
    \end{itemize}
\end{theorem}

\begin{proof}
\textit{(i)}
Recall that 
	\begin{align*}
		| \Pi^{\ast, s}_p  f | = \dfrac{1}{n} \int_{\Sphere} \| \xi \|^{-n}_{\Pi^{\ast, s}_p  f} \, d\xi.
	\end{align*}
Owing to Proposition~\ref{prop.exist.infi.polar}--\textit{(i)}, the family of functions $ \big\{ \| \cdot \|^{-n}_{\Pi^{\ast, s}_p \, f} \big\}_{p \geq 1}$ is uniformly bounded on $\Sphere$ for every sufficiently large $p$.
Furthermore, Proposition~\ref{prop.lim.gauge}--\textit{(i)} states the pointwise convergence: $\textstyle\lim_{p \to \infty} \| \xi \|_{\Pi^{\ast, s}_p \, f}^{-n} = \| \xi \|_{\Pi^{\ast, s}_\infty \, f}^{-n}$ for a.e $\xi \in \Sphere$.
Therefore, using the Lebesgue dominated convergence theorem, we conclude 
\begin{align*}
	\lim_{p \to \infty} | \Pi^{\ast, s}_p f | = \dfrac{1}{n} \int_{\Sphere} \| \xi \|^{-n}_{\Pi^{\ast, s}_\infty f} \, d\xi = | \Pi^{\ast, s}_\infty f |.
\end{align*}

\medskip

The arguments for \textit{(ii)} and \textit{(iii)} are analogous, which completes the proof of Theorem \ref{thm.volume}.
\end{proof}

Using the definition \ref{vtil} together with similar arguments as in the proof of Theorem~\ref{thm.volume}, we obtain the following limiting behavior for the dual mixed volume pf $\Pi^{\ast, s}_p f$. 

\begin{theorem}\label{thm.mixed}
    Let $f \in W^{1, 1}(\R^n) \cap W^{1, \infty}(\R^n)$ be nonzero.
	Assume that one of the following conditions is fulfilled:
	\begin{itemize}
		\item[(a)] $K$ is a star body and $q > n$;
		
		\item[(b)] $K$ is a bounded star body and $q \in (- \infty, n) \setminus \{ 0 \}$.
    \end{itemize}
    Then, the following limits hold:
    \begin{itemize}
        \item[(i)] for any fixed $s \in (0, 1)$, $\textstyle\lim_{p \to \infty} \vtil_{q}(K, \Pi^{\ast, s}_p f) = \vtil_{q}(K, \Pi^{\ast, s}_\infty  f)$;
        
        \item[(ii)] $\textstyle\lim_{p \to + \infty} \vtil_{q}(K, \Pi^{\ast}_p \, f) = \vtil_{q}(K, \Pi^{\ast}_\infty \, f)$;
        
        \item[(iii)]  $\textstyle\lim_{s \to 1^-} \vtil_{q}(K, \Pi^{\ast, s}_\infty  f) = \vtil_{q}(K, \Pi^{\ast}_\infty f)$.
    \end{itemize}
\end{theorem}

\begin{proof}[Proof of Theorem~\ref{thm.mixed}]
\textit{(i)} Recall  that for any $q \not\in \{ 0, n \}$, 
\begin{align*}
	\vtil_{q}(K,  \Pi^{\ast, s}_p  f) = \dfrac{1}{n} \int_{\Sphere} \| \xi \|^{- n + q}_K \| \xi \|^{-q}_{\Pi^{\ast, s}_p f} \, d\xi.
\end{align*}
Fix $s \in (0, 1)$.
Let us proceed by considering two cases.

\smallskip

\textbf{Case 1:} \textit{$K$ is a star body and $q > n$}. 
Thanks to the continuity of $\xi \mapsto \| \xi \|_K$, the compactness of $\Sphere$ implies that the map $\xi \mapsto \| \xi \|_K^{- n + q}$ is bounded on $\Sphere$.
Further, Proposition~\ref{prop.exist.infi.polar}--\textit{(i)} implies that the family $\big\{ \| \cdot \|_{\Pi^{\ast, s}_p \, f}^{-q} \big\}_{p \geq 1}$ is uniformly bounded on the unit sphere.
Therefore, applying the Lebesgue dominated convergence theorem, together with Proposition~\ref{prop.lim.gauge}--\textit{(i)}, we deduce 
\begin{align*}
	\lim_{p \to \infty} \vtil_{q}(K, \Pi^{\ast, s}_p f) = & ~ \frac{1}{n} \int_{\Sphere} \lim_{p \to \infty} \| \xi \|^{- n + q}_K \| \xi \|^{-q}_{\Pi^{\ast, s}_p \, f} \, d\xi \\
    = & ~  \frac{1}{n} \int_{\Sphere} \| \xi \|^{- n + q}_K \| \xi \|^{-q}_{\Pi^{\ast, s}_\infty f} \, d\xi =  \vtil_{q}(K, \Pi^{\ast, s}_\infty  f).
\end{align*}

\medskip

\textbf{Case 2:} \textit{ $K$ is a bounded star body and $q \in (- \infty, n) \setminus \{ 0 \}$.}
Using the fact that $K$ is a bounded star body and $q < n$, we infer that the map $\xi \mapsto \| \xi \|_K^{- n + q}$ is bounded on $\Sphere$.
If $q \in (0, n)$, the proof proceeds exactly as in Case 1.
If $q < 0$, Remark~\ref{rem.bounded-gauge}--\textit{(i)} implies that, for $\bar p > 1$ sufficiently large, the family $\big\{ \| \cdot \|_{\Pi^{\ast, s}_p \, f}^{-q} \big\}_{p \geq \bar p}$ is uniformly bounded almost everywhere on the unit sphere.
Finally, by the Lebesgue dominated convergence theorem, together with Proposition~\ref{prop.lim.gauge}--\textit{(i)}, the conclusion follows.

\medskip

We then can prove \textit{(ii)} and \textit{(iii)} by using similar arguments.
This completes the proof of Theorem~\ref{thm.mixed}.
\end{proof}

To continue, let us introduce the notion of $s$--scaled dilation factor.

\begin{definition}\label{def.dila}
	Let $s \in (0, 1]$ and let $K$ and $L$ be two star bodies.
	The $s$--scaled dilation factor of $K$ relative to $L$ is defined by
	\[
		\dila_s(K, L) := \sup_{\xi \in \Sphere} \| \xi \|_K^{-s} \| \xi \|_L^s = \sup_{\xi \in \Sphere} \rho_K(\xi)^s \rho_L(\xi)^{-s}.
	\]
	If $s = 1$, we simply write $\dila(K, L) := \dila_1(K, L)$.
\end{definition}

We remark that the dilation factor $\dila(K, L)$ quantifies how much one has to dilate $L$ to contain $K$:
\[
	\dila(K, L) = \inf \big\{ \lambda > 0 : K \subset \lambda L \big\}.
\]

\medskip

Our second main result, Theorem \ref{thm.mixed-dila}, establishes the asymptotic behavior of $\vtil_{-sq} (K, \Pi^{\ast, s}_p \, f)^{1/p}$. 
Its proof is an adaptation of the argument in Lemma \ref{prop.lim.gauge} and we include it here for completeness.

\begin{theorem}\label{thm.mixed-dila}
	Let $f \in W^{1, 1}(\R^n) \cap W^{1, \infty}(\R^n)$ be nonzero and let $K$ be a bounded star body.
	Then, the following limits hold: 
	\begin{itemize}
		\item[(i)] for any fixed $s \in (0, 1)$, one has $\textstyle\lim_{p \to \infty} \vtil_{-sp} (K, \Pi^{\ast, s}_p \, f)^{1/p} = \dila_s(K, \Pi^{\ast, s}_\infty \, f)$;
		
		\item[(ii)] $\textstyle\lim_{p \to \infty} \vtil_{-p}(K, \Pi^\ast_p \, f)^{1/p} = \dila(K, \Pi^\ast_\infty \, f)$;
		
		\item[(iii)] $\textstyle\lim_{s \to 1^-} \dila_s(K, \Pi^{\ast, s}_\infty \, f) = \dila(K, \Pi^\ast_\infty \, f)$.
	\end{itemize}
\end{theorem}

\begin{proof}
Since the proofs of \textit{(i)} and \textit{(ii)} are similar, we present only the proof of \textit{(i)}.

\medskip

\textit{(i)}
Recall that the dual mixed volume $\vtil_{-sp}(K, \Pi^{\ast, s}_p \, f)$ can be expressed in terms of the anisotropic Gagliardo seminorm (see e.g. \cite[Page 1098]{HL_2024}):
	\begin{align*}
		n \vtil_{-sp}(K, \Pi^{\ast, s}_p \, f) = 
		\int_{\Sphere} \| \xi \|_K^{- n - ps} \| \xi \|_{\Pi^{\ast, s}_p \, f}^{ps} \, d\xi =  p(1 - s) \int_{\R^n} \int_{\R^n} \dfrac{|f(x) - f(y)|^p}{\| x - y \|_K^{n + sp}} \,dx \,dy.
	\end{align*}
Using the inclusion \eqref{inclu-01} and the fact that $K$ is a bounded star body, we first observe that
	\begin{equation}\label{fini.semi}
		\int_{\R^n} \int_{\R^n} \dfrac{|f(x) - f(y)|}{\| x - y \|_K^{n + s}} \, dx \, dy < + \infty.
	\end{equation}

\smallskip

On the one hand, for any fixed $p > 1$, we have the following estimate 
\begin{equation}\label{est.seminorm-01}
	\begin{split}
    \frac{n}{p (1 - s)} \widetilde V_{-sp}(K, \Pi^{\ast, s}_p \, f) = & ~ 
     \int_{\R^n} \int_{\R^n} \dfrac{|f(x) - f(y)|^{p - 1}}{\| x - y \|^{s(p - 1)}_K} \dfrac{|f(x) - f(y)|}{\| x - y \|^{n + s}_K} \, dx\,dy
     \\
    \leq & ~ 
   \dila_s(K, \Pi^{\ast, s}_\infty \, f)^{p - 1} \int_{\R^n} \int_{\R^n}  \dfrac{|f(x) - f(y)|}{\| x - y \|_K^{n + s}} \, dx \, dy,
   \end{split}
\end{equation}
where we have used the identity
\begin{equation}\label{dila-holder}
	\dila_s(K, \Pi^{\ast, s}_\infty \, f) = \sup_{\xi \in \Sphere} \| \xi \|^{-s}_K \| \xi \|_{\Pi^{\ast, s}_\infty \, f}^s = \sup_{x \neq y} \dfrac{|f(x) - f(y)|}{\| x - y \|_K^s}.
\end{equation}
Taking  first the $p$th root in the inequality~\eqref{est.seminorm-01} and then letting $p$ tend to $\infty$, thanks to \eqref{fini.semi}, we get
    \begin{equation}\label{est.vtil-dila-01}
        \limsup_{p \to \infty} \vtil_{-sp}(K, \Pi^{\ast, s}_p \, f)^{1/p} \leq \dila_s(K, \Pi^{\ast, s}_\infty \, f).
    \end{equation}
    
\medskip

On the other hand, for any fixed $\varepsilon \in (0, \dila_s(K, \Pi^{\ast, s}_\infty \, f))$, denote
    \begin{equation}\label{A_hat}
        \widehat A_{s, \varepsilon} 
        :=
		\left\{ 
		(x, y) \in \R^n \times \R^n :
		\dfrac{|f(x) - f(y)|}{\| x - y \|_K^s}
		> 
		\dila_s(K, \Pi^{\ast, s}_\infty f) - \varepsilon
		\right\}.
   \end{equation}
Notice that $|\widehat{A}_{s, \varepsilon}| > 0$.
By direct computation, we obtain
    \begin{equation}\label{est.A_hat}
    	\begin{split} 
         & ~  \dfrac{n}{p(1 - s)} \widetilde V_{-sp}(K, \Pi^{\ast, s}_p \, f) \\
          \geq &  ~  
          \iint_{\widehat A_{s, \varepsilon}} \dfrac{|f(x) - f(y)|^p}{\| x - y \|_K^{n + sp}} \, dx \, dy \\
         \geq & ~
         \left( \dila_s(K, \Pi^{\ast, s}_\infty \, f) - \varepsilon \right)^{p - 1} 
         \iint_{\widehat A_{s,\varepsilon}}\dfrac{\| x - y \|^{s(p - 1)}_K |f(x) - f(y)|}{\| x - y \|^{n + sp}_K } \,dx\,dy  \\
         = & ~ 
        \left( \dila_s(K, \Pi^{\ast, s}_\infty \, f) - \varepsilon \right)^{p - 1} 
        \iint_{\widehat A_{s, \varepsilon}} \dfrac{|f(x) - f(y)|}{\|x - y\|^{n + s}_K } \,dx\,dy.
         \end{split}
    \end{equation}
Therefore, taking first the $p$th root of the above inequality and letting $p$ tend to $\infty$, we arrive at
\[
        \liminf_{p \to \infty} \vtil_{-sp}(K, \Pi^{\ast, s}_p \, f)^{1/p} \geq \dila_s(K, \Pi^{\ast, s}_\infty \, f) - \varepsilon.
    \]
Since $\varepsilon > 0$ is arbitrary, it follows that
    \begin{equation}\label{est.vtil-dila-02}
        \liminf_{p \to \infty} \vtil_{-sp}(K, \Pi^{\ast, s}_p \, f)^{1/p} \geq \dila_s(K, \Pi^{\ast, s}_\infty \, f).
    \end{equation}
Combining \eqref{est.vtil-dila-01} and \eqref{est.vtil-dila-02}, we conclude that 
\[
\textstyle\lim_{p \to \infty} \vtil_{-sp}(K, \Pi^{\ast, s}_p \, f)^{1/p} = \dila_s(K, \Pi^{\ast, s}_\infty \, f).
\]

\medskip

\textit{(iii)}
Fix $\delta > 0$ such that
\[
    \delta < \dfrac{\dila(K, \Pi^\ast_\infty \, f)}{1 + \dila(K, \Pi^\ast_\infty \, f)}.
\]
Then, there exist $x_o \in \R^n$, $\xi_o \in \Sphere$ and $t_o > 0$
such that 
    \begin{equation}\label{dila-est-01}
    \begin{split}
	(1 - \delta) \, \dila(K, \Pi^\ast_\infty f)
	\leq & ~ \frac{|\langle \nabla f(x_o), \xi_o \rangle|}{\|\xi_o\|_K}
	\\
    = & ~ \lim_{t \to 0^+} \frac{|f(x_o + t \xi_o) - f(x_o)|}{t \| \xi_o \|_K }
	\leq \frac{|f(x_o + t_o \xi_o) - f(x_o)|}{t_o \| \xi_o \|_K } + \delta.
    \end{split}
	\end{equation}
Furthermore, a direct computation yields
    \begin{equation}\label{dila-est-02}
    \begin{split}
        \frac{|f(x_o + t_o \xi_o) - f(x_o)|}{t_o \| \xi_o \|_K }
        = & ~
        (t_o \| \xi_o \|_K)^{s - 1} \dfrac{|f(x_o + t_o \xi_o) - f(x_o)|}{\|t_o \xi_o \|_K^s}
        \\
        \leq & ~  (t_o \| \xi_o \|_K)^{s - 1} \dila_s(K, \Pi^{\ast, s}_\infty \, f),
    \end{split}
    \end{equation}
for every $s \in (0, 1)$.
Combining \eqref{dila-est-01} and \eqref{dila-est-02}, we obtain
    \begin{align*}
        (1 - \delta) \dila(K, \Pi^\ast_\infty \, f) - \delta
        \leq
         (t_o \| \xi_o \|_K)^{s - 1} \dila_s(K, \Pi^{\ast, s}_\infty \, f) \quad \text{ for every } s \in (0, 1).
    \end{align*}
Letting first $s \to 1^-$ and then $\delta \searrow 0$, we deduce
    \begin{equation}\label{liminf-dila}
        \dila(K, \Pi^\ast_\infty \, f) \leq \liminf_{s \to 1^-}  \dila_s(K, \Pi^{\ast, s}_\infty \, f) \, .
    \end{equation}
    
\medskip

It remains to check the inequality $\dila(K, \Pi^\ast_\infty \, f) \geq \textstyle\limsup_{s \to 1^-} \dila_s(K, \Pi^{\ast, s}_\infty \, f)$.
Without loss of generality, assume that that $\textstyle\limsup_{s \to 1^-} \dila_s(K, \Pi^{\ast, s}_\infty \, f) > 0$ and hence there exist $\bar s \in (0, 1)$ such that $\dila_s(K, \Pi^{\ast, s}_\infty \, f) > 0$ for every $s \in (\bar s, 1)$.
Let us fix $s \in (\bar s, 1)$ and $\varepsilon \in (0, \dila_s(K, \Pi^{\ast, s}_\infty \, f))$.
Let $\widehat  A_{s, \varepsilon}$ be defined as in \eqref{A_hat}.

\medskip

For any fixed $p > 1$, using a simple change of variables and the estimate \eqref{est.A_hat}, we obtain
	\begin{equation}\label{esti.dila-01}
		\begin{split}
		\dfrac{1}{p(1 - s)}\int_{\Sphere} \| \xi \|_K^{-n - ps} \| \xi \|^{ps}_{\Pi^{\ast, s}_p f} \, d\xi = & ~ \dfrac{n}{p(1 - s)} \widetilde V_{-sp}(K, \Pi^{\ast, s}_p f) \\
		\geq & ~ 
		\left( \dila_s(K, \Pi^{\ast, s}_\infty  f) - \varepsilon \right)^{p - 1} 
		\iint_{\widehat A_{s, \varepsilon}} 
		\dfrac{|f(x) - f(y)|}{\| x - y \|_K^{n + s}} \, dx \, dy.
		\end{split}
	\end{equation}
Thanks to the estimate \eqref{esti.iii-02}, for every $\varrho > 0$ and $s \in (0, 1)$, one has
	\begin{equation}\label{esti.dila-02}
    \begin{split}
		& ~ \dfrac{1}{p(1 - s)}\| \xi \|^{sp}_{\Pi^{\ast, s}_p f} \\
		\leq & ~ 
        \| \xi \|_{\Pi^\ast_\infty \, f}^{s(p - 1)} 
		\left(
		\dfrac{\varrho^{p(1 - s)}}{p(1 - s)} \| \nabla f \|^{(p - 1)(1 - s)}_{L^\infty(\R^n)} \| \nabla f \|_{L^1(\R^n)} 
		+
		\dfrac{1}{s \varrho^s} \| 2f \|^{(p - 1)(1 - s)}_{L^\infty(\R^n)} \| 2f \|_{L^1(\R^n)} 
		\right).
    \end{split}
	\end{equation}
Notice that the function $r \mapsto r^{1/p}$ is concave for $p > 1$.
Since $K$ and $\Pi^{\ast}_\infty \, f$ are bounded star bodies, there exists a constant $c_o > 0$ such that $\| \xi \|_K^{-n} \| \xi \|^{-s}_{\Pi^{\ast}_\infty \, f} \leq c_o$ for every $\xi \in  \Sphere$.
Therefore, it follows from \eqref{esti.dila-01} and \eqref{esti.dila-02} that
	\begin{equation}
		\begin{split}
		&~ B_{\varrho, s, p} \left( \int_{\Sphere} \| \xi \|_K^{-ps} \| \xi \|^{ps}_{\Pi^\ast_\infty \, f} d\xi \right)^{1/p} \\
		\geq & ~ 
		\left( \dila_s(K, \Pi^{\ast, s}_\infty  \, f) - \varepsilon \right)^{1 - 1/p} 
		\left(  
			\iint_{\widehat A_{s, \varepsilon}} 
		\dfrac{|f(x) - f(y)|}{\| x - y \|_K^{n + s}} \, dx \, dy
		\right)^{1/p},
		\end{split}
	\end{equation}
where
	\[
		B_{\varrho, s, p}  := c_o^{1/p} 
		\Bigg(  
			\dfrac{\varrho^{(1 - s)}}{(p(1 - s))^{1/p}}
			\| \nabla f \|^{\left(1 - 1/p \right)(1 - s)}_{L^\infty(\R^n)} \| \nabla f \|_{L^1(\R^n)}^{1/p} 
		+
		\dfrac{1}{(s \varrho^s/2)^{1/p} } \| 2f \|^{\left(1 - 1/p\right)(1 - s)}_{L^\infty(\R^n)} \| 2f \|_{L^1(\R^n)}^{1/p}
		\Bigg).
	\]
Letting $p$ tend to $\infty$, we arrive at
	\begin{equation}
		\left( \varrho^{1 - s} \| \nabla f \|_{L^{\infty}(\R^n)}^{1 - s} + \| 2f \|_{L^{\infty}(\R^n)}^{1 - s} \right)
		\underbrace{\sup_{\xi \in \Sphere} \| \xi \|_K^{-s} \| \xi \|_{\Pi^\ast_\infty \, f}^s}_{= \, \dila(K, \Pi^{\ast}_\infty \, f)^s} 
		\geq 
		\dila_s(K, \Pi^{\ast, s}_\infty  \, f) - \varepsilon.
	\end{equation}
Finally, letting first $\varrho \searrow 0$ and then $\varepsilon \searrow 0$, the above inequality implies that
	\begin{equation}\label{limsup-dila}
		\limsup_{s \to 1^-} \dila_s(K, \Pi^{\ast, s}_\infty  \, f) 
		\leq 
		\dila(K, \Pi^{\ast}_\infty \, f).
	\end{equation}
Collecting \eqref{liminf-dila} and \eqref{limsup-dila}, we finally conclude that $\textstyle\lim_{s \to 1^-} \dila_s(K, \Pi^{\ast, s}_\infty  \, f)  = \dila(K, \Pi^{\ast}_\infty \, f)$, which completes the proof of Theorem \ref{thm.mixed-dila}.
\end{proof}







\section{Asymmetric analogues and geometric inequalities}\label{sec:asymmetric}

In this final section, we discuss our results for the asymmetric variants of the fractional $L^p$ polar projection body.
Furthermore, we derive related P\'olya--Szeg\H{o} and functional isoperimetric--type inequalities within the $L^\infty$ framework.

\subsection{Asymmetric ($s$--fractional) $L^\infty$ polar projection bodies}

Asymmetric (fractional) $L^p$ polar projection bodies refine the classical construction via the use of positive and negative parts (instead of the absolute value) in the definition of gauge functions.
This modification yields sharper affine Sobolev inequalities, see~\cite{HS_2009, HL_2024}.
The main contributions of this section, Propositions~\ref{prop.lim.pm}--\ref{prop.limit.1/p.pm}, establish the asymptotic behavior of the asymmetric fractional $L^p$ polar projection bodies.

\medskip

From now on, for any real number $a \in \R$, we denote $a_+ = \max \{ a, 0 \}$ and $a_- = \max \{ -a , 0 \}$.

\begin{definition}Fix  $p \in [1, \infty)$ and $s \in (0, 1)$.
    \begin{itemize}
    		\item[$\bullet$] For any $f \in W^{1, p} (\R^n)$, its asymmetric $L^p$ polar projection body $\Pi^{\ast}_{p, +} \, f$ (respectively  $\Pi^{\ast}_{p, -} \, f$) is a star--shaped set defined via the gauge function, for every $\xi \in \R^n$
		\begin{align*}
		\| \xi \|_{\Pi^*_{p, +} \, f}^p := \int_{\R^n} \langle \nabla f(x), \xi \rangle_+^p dx
			\left( \text{respectively } 
            \| \xi \|_{\Pi^*_{p, -} \, f}^p := \int_{\R^n} \langle \nabla f(x), \xi \rangle_-^p dx 
            \right).
	\end{align*}
	
		\item[$\bullet$] 
        For a fixed $f \in W^{1, \infty}(\R^n)$, its asymmetric $s$--fractional $L^\infty$ polar projection body $\Pi^{\ast}_{\infty, +} \, f$ (respectively $\Pi^{\ast}_{\infty, -}$) is a star--shaped set defined via the gauge function, for every $\xi \in \R^n$
        \begin{align*}
            \| \xi \|_{\Pi^{\ast}_{\infty, +} \, f} 
            := \esssup_{x \in \R^n} 
            \, \langle \nabla f(x), \xi \rangle_+  \left( \text{respectively } 
            \| \xi \|_{\Pi^{\ast}_{\infty, -} \, f} 
            := \esssup_{x \in \R^n} 
            \, \langle \nabla f(x), \xi \rangle_- \right).
        \end{align*}

		\item[$\bullet$] For any $f \in W^{s, p}(\R^n)$, its asymmetric $s$--fractional $L^p$ polar projection body $\Pi^{\ast, s}_{p, +} \, f$ (resp $\Pi^{\ast, s}_{p, -} \, f$) is a star--shaped set defined via the gauge function, for every $\xi \in \R^n$
		\begin{align*}
		\| \xi \|^{ps}_{\Pi^{\ast, s}_{p, +} f} := p(1 - s)\int_0^\infty t^{-ps - 1} \int_{\R^n} (f(x + t\xi) - f(x))_+^p \, dx \, dt 
		\\
		\left( \text{respectively } 
            \| \xi \|^{ps}_{\Pi^{\ast, s}_{p, -} f} := p(1 - s)\int_0^\infty t^{-ps - 1} \int_{\R^n} (f(x + t\xi) - f(x))_-^p \, dx \, dt             
         \right).
		\end{align*}
	
	\item[$\bullet$] For a fixed $f \in W^{s, \infty}(\R^n)$, its asymmetric $s$--fractional $L^\infty$ polar projection body $\Pi^{\ast, s}_{\infty, +} \, f$ (respectively $\Pi^{\ast, s}_{\infty, -}$) is a star--shaped set defined via the gauge function, for every $\xi \in \R^n$
        \begin{align*}
            \| \xi \|_{\Pi^{\ast, s}_{\infty, +} \, f}^s 
            := \sup_{(x, t) \in \R^n \times (0, \infty)} 
            \dfrac{1}{t^s} (f(x + t\xi) - f(x))_+ \\
            \left( \text{respectively } 
            \| \xi \|_{\Pi^{\ast, s}_{\infty, -} \, f}^s 
            := \sup_{(x, t) \in \R^n \times (0, \infty)} 
            \dfrac{1}{t^s} (f(x + t\xi) - f(x))_- \right).
        \end{align*}
    \end{itemize}
\end{definition}

Under these definitions, we can obtain results analogous to Proposition \ref{prop.exist.infi.polar} and Lemma \ref{prop.lim.gauge}.
However, to derive lower bounds for the gauge functions similar to those in Proposition \ref{prop.exist.infi.polar} for asymmetric fractional $L^p$ polar projection bodies, some additional technical details are required.
These are provided in detail below.

\begin{lemma}\label{lem.exist.pm}
Let $f \in W^{1, 1}(\R^n) \cap W^{1, \infty}(\R^n)$ be nonzero.
The following assertions hold true:
	 \begin{itemize}
        \item[(i)] 
        For any fixed $s \in (0, 1)$, there exist $c = c(s, f) > 0$ and $\bar p = \bar p(s, f) > 1$ sufficiently large such that $\Pi^{\ast, s}_{p, \pm} \, f \subset c B^n$ for every $p \geq \bar p$. 
        
        \item[(ii)] 
        For any fixed $s \in (0, 1)$, the sets $\Pi^{\ast, s}_{\infty, +} \, f$ and $\Pi^{\ast, s}_{\infty, -} \, f$ are star bodies with the origin in their interior.
        Furthermore, there exists $c = c(f) > 0$ and $\bar s = \bar s(f) \in (0, 1)$ such that $\Pi^{\ast, s}_{\infty, \pm} \, f \subset c B^n$ for every $s \in (\bar s, 1)$.

        \item[(iii)]
        The sets $\Pi^{\ast}_{\infty, +} \, f$ and $\Pi^{\ast}_{\infty, -} \, f$ are star bodies with the origin in their interiors.
    \end{itemize}
\end{lemma}

\begin{proof}
In what follows, we provide lower bounds for the gauge functions of $\Pi^{\ast, s}_{p, \pm} \, f$ in \textit{(i)} and $\Pi^{\ast, s}_{\infty, \pm} \, f$ in~\textit{(ii)}.
These arguments differ from those used in the proof of Proposition~\ref{prop.exist.infi.polar}. 
The remaining steps of the proof can be carried out analogously to those in Proposition~\ref{prop.exist.infi.polar}.

\medskip

\textit{(i)}
It suffices to check the inclusion $\Pi^{\ast, s}_{p, +} \, f \subset c B^n$.
For any fixed ${\alpha \in (0, 1)}$, set
\begin{align*}
	\Sset_{\alpha} := \left\{ x \in \R^n : |f(x)| \geq \alpha \| f \|_{L^\infty(\R^n)} \right\}
	\quad \text{ and } \quad 
	\Sset_{\alpha}^\pm := \left\{ x \in \R^n : f(x)_{\pm} \geq \alpha \| f \|_{L^\infty(\R^n)} \right\}.
\end{align*}
Note that $\Sset_{\alpha}^{+} \cup \Sset_{\alpha}^{-} = \Sset_{\alpha}$.
Thanks to Claim \ref{claim.W_alpha_bounded}, $\Sset_\alpha$ is bounded and hence $\Sset_{\alpha}^{\pm}$  are also bounded.
Let $r_o > 1$ be such that $\Sset_{3/4} \subset \Sset_{1/4} \subset r_o B^n$.
We first observe that
	\[
		\| f_+ \|_{L^p(\Sset_{3/4}^+)} \geq \dfrac{3}{4} \| f \|_{L^\infty(\R^n)} | \Sset_{3/4}^+ |^{1/p} 
		\quad \text{ and } \quad 
		\| f_- \|_{L^p(\Sset_{3/4}^-)} \geq \dfrac{3}{4} \| f \|_{L^\infty(\R^n)} | \Sset_{3/4}^- |^{1/p}.
	\]
Fix $\xi \in \Sphere$.
Due to the choice of $r_o$, for any $t \in \R$ with $|t| > 2r_o$, the sets $r_o B^n$ and $\Sset_{3/4} + t \xi$ are disjoint (and hence \(r_o B^n\) and \(\Sset_{3/4}^{\pm} + t \xi\) are also disjoint).
Therefore, we have
	\begin{align*}
		\| f_+ \|_{L^p(\Sset_{3/4}^+ + t\xi)} \leq \dfrac{1}{4} \| f \|_{L^\infty(\R^n)} | \Sset_{3/4}^+ |^{1/p}
		\quad \text{ and } \quad 
		\| f_- \|_{L^p(\Sset_{3/4}^- + t \xi)} \leq \frac{1}{4} \| f \|_{L^\infty(\R^n)} | \Sset_{3/4}^- |^{1/p}.
	\end{align*}
Notice that for any measurable set $U \subset \R^n$ and any functions $g, h \in L^p(U)$, it holds 
\[
\| (g - h)_\pm \|_{L^p(U)} \geq \| g_\pm \|_{L^p(U)} - \| h_\pm \|_{L^p(U)}.
\]
Fix $t > 2r_o$.
On the one hand, we have
\begin{align*}
	\| (f(\cdot + t\xi) - f )_+ \|_{L^p(\Sset_{3/4}^-)} 
    = & ~ 
    \| (f - f(\cdot + t \xi))_- \|_{L^p(\Sset_{3/4}^-)} \\
    \geq & ~ \| f_- \|_{L^p(\Sset^-_{3/4})} -  \underbrace{\| f(\cdot + t\xi)_- \|_{L^p(\Sset^-_{3/4})} }_{\, = \, \| f_- \|_{L^p(\Sset^-_{3/4} + t\xi)} } \\
	\geq & ~ \dfrac{1}{2} \| f \|_{L^\infty(\R^n)}|\Sset^-_{3/4}|^{1/p}.
\end{align*} 
On the other hand, estimating over the set $\Sset^+_{3/4} - t\xi$ gives
\begin{align*}
	\| (f(\cdot + t\xi) - f )_+ \|_{L^p(\Sset_{3/4}^+ - t\xi)} = & ~ \| (f - f(\cdot - t\xi))_+ \|_{L^p(\Sset_{3/4}^+)} \\
	\geq & ~ 
	\|	f_+ \|_{L^p(\Sset_{3/4}^+)} - \underbrace{\| f(\cdot - t \xi)_+ \|_{L^p(\Sset_{3/4}^+)} }_{ \, = \,  \| f_+ \|_{L^p(\Sset_{3/4}^+ - t \xi)}} 
	\\ 
	\geq & ~ \dfrac{1}{2} \| f \|_{L^\infty(\R^n)} |\Sset^+_{3/4}|^{1/p}.
\end{align*}
Collecting the above estimates, we arrive at
\begin{align*}
	\| (f(\cdot + t\xi) - f)_+ \|_{L^p(\R^n)} \geq \dfrac{1}{4} \| f \|_{L^\infty(\R^n)} (|\Sset^+_{3/4}|^{1/p} + |\Sset^-_{3/4}|^{1/p}) \geq \dfrac{1}{4} \| f \|_{L^\infty(\R^n)} |\Sset_{3/4}|^{1/p} > 0,
\end{align*}
for every $t > 2r_o$.
We proceed as in the proof of Proposition \ref{prop.exist.infi.polar} to obtain
\begin{align*}
	\| \xi \|^s_{\Pi^{\ast, s}_{p, +} \, f} \geq \dfrac{1}{4} (2r_o)^{-s}\| f \|_{L^\infty(\R^n)} |\Sset_{3/4}|^{1/p} \left( \dfrac{1 - s}{s} \right)^{1/p}.
\end{align*}
Consequently, we are able to concldue that there exist $c = c(s, f) > 0$ and $\bar p = \bar p(s, f) > 1$ satisfying $\Pi^{\ast, s}_{p, +} \, f \subset c B^n$ for every $p \geq \bar p$.

\medskip

\textit{(ii)}
We will provide a positive lower bound for $\| \cdot \|_{\Pi^{\ast, s}_{\infty, \pm} \, f}$ on $\Sphere$; again, it is sufficient to prove for $\Pi^{\ast, s}_{\infty, +}$.
Let $r > 1$ be such that
	\begin{align*}
		\| f_+ \|_{L^1(rB^n)} \geq \dfrac{2}{3} \| f_+ \|_{L^1(\R^n)} \quad \text{ and } \quad \| f_- \|_{L^1(rB^n)} \geq \dfrac{2}{3} \| f_- \|_{L^1(\R^n)}.
	\end{align*}
Fix $\xi \in \Sphere$.
For any $|t| > 2r$, two sets $rB^n$ and $rB^n - t\xi$ are disjoint and therefore we have
	\begin{align*}
		\| f_+ \|_{L^1(rB^n + t\xi)} \leq \dfrac{1}{3} \| f_+ \|_{L^1(\R^n)} \quad \text{ and } \quad \| f_- \|_{L^1(rB^n + t\xi)} \leq \dfrac{1}{3} \| f_- \|_{L^1(\R^n)}.
	\end{align*}
With these remarks, using triangle inequality, we obtain the following estimates for any $t \geq 2r$:
    \begin{align*}
    	\| (f(\cdot + t\xi) -  f)_+ \|_{L^\infty(rB^n - t\xi)} \geq & ~
	\dfrac{1}{|rB^n - t\xi|} \| (f(\cdot + t\xi)  - f)_+ \|_{L^1(rB^n - t\xi)}  \\
	= & ~ 
	\dfrac{1}{r^n |B^n|} \| (f - f(\cdot  - t\xi))_+ \|_{L^1(rB^n)} 
	\\
	\geq & ~ 	
	\dfrac{1}{r^n |B^n|} \Big( 
	\underbrace{  \| f_+ \|_{L^1(rB^n)} -  \| f(\cdot  - t\xi)_+ \|_{L^1(rB^n)} }_{
	\, = \, \| f_+ \|_{L^1(rB^n)} -  \| f_+ \|_{L^1(rB^n - t\xi)} 
	}
    \Big)
	\\
	\geq & ~ 
	\dfrac{1}{3 r^n |B^n|} \| f_+ \|_{L^1(\R^n)}
    \end{align*}
    and
    \begin{align*}
    	\| (f(\cdot + t\xi) -  f)_+ \|_{L^\infty(rB^n)} \geq & ~
	\dfrac{1}{|rB^n|} \| (f(\cdot + t\xi)  - f)_+ \|_{L^1(rB^n)}  \\
	= & ~ 
	\dfrac{1}{r^n |B^n|} \| (f - f(\cdot  + t\xi))_- \|_{L^1(rB^n)} 
	\\
	\geq & ~ 	
	\dfrac{1}{r^n |B^n|} 
	\Big( \underbrace{  \| f_- \|_{L^1(rB^n)} -  \| f(\cdot  + t\xi)_- \|_{L^1(rB^n)} }_{
	\, = \, \| f_- \|_{L^1(rB^n)} -  \| f_- \|_{L^1(rB^n + t\xi)} 
	}
    \Big)
	\\
	\geq & ~ 
	\dfrac{1}{3 r^n |B^n|} \| f_- \|_{L^1(\R^n)}.
    \end{align*}
Therefore, we obtain
	\begin{align*}
		\| (f(\cdot + t\xi) - f)_+ \|_{L^\infty(\R^n)} \geq \dfrac{1}{6 r^n |B^n|} \big( \| f_+ \|_{L^1(\R^n)} + \| f_- \|_{L^1(\R^n)} \big) = \dfrac{1}{6 r^n |B^n|} \| f \|_{L^1(\R^n)}. 
	\end{align*}
Using the definition $\| \cdot \|_{\Pi^{\ast, s}_\infty \, f}$, we are led to
	\begin{align*}
		\| \xi \|_{\Pi^{\ast, s}_\infty \, f} = \left( \sup_{t > 0} \dfrac{1}{t^s} \| (f(\cdot + t\xi) - f)_+ \|_{L^\infty(\R^n)} \right)^{1/s} \geq \dfrac{\upsilon^{1/s}}{2r} 
		\quad \text{ with } \quad
		\upsilon = \dfrac{\| f \|_{L^1(\R^n)}}{6r^n |B^n|}.
	\end{align*}
This implies the fact that there exists $c = c(f) > 0$ and $\bar s = \bar s(f) \in (0, 1)$ such that $\Pi^{\ast, s}_{\infty, \pm} \, f \subset c B^n$ for every $s \in (\bar s, 1)$.
Lemma \ref{lem.exist.pm} is proven.
\end{proof}

Next, one can obtain results  analogous to Proposition \ref{prop.lim.gauge} for the gauge functions of asymmetric fractional $L^p$ polar projection bodies.
Consequently, we get the asymptotic behavior of the associated volume and dual mixed volumes stated in the following propositions, whose proofs are omitted since they are similar to those in Theorem \ref{thm.volume}-\ref{thm.mixed} and Theorem \ref{thm.mixed-dila}.

\begin{proposition}\label{prop.lim.pm}
    Let $f \in W^{1, 1}(\R^n) \cap W^{1, \infty}(\R^n)$ be nonzero.
    Assume that one of the following conditions is fulfilled:
	\begin{itemize}
		\item[(a)] $K$ is a star body and $q > n$;
		
		\item[(b)] $K$ is a bounded star body and $q \in (- \infty, n) \setminus \{ 0 \}$.
    \end{itemize}
    Then, the following limits hold:
    \begin{itemize}
        \item[(i)] for any fixed $s \in (0, 1)$, one has
        \[
        \textstyle\lim_{p \to \infty} |\Pi^{\ast, s}_{p, \pm} \, f| =  |\Pi^{\ast, s}_{\infty, \pm} \, f|
        \text{ and } 
        \textstyle\lim_{p \to \infty} \vtil_{q}(K, \Pi^{\ast, s}_{p, \pm} \, f) = \vtil_q(K, \Pi^{\ast, s}_{\infty, \pm} \, f);
        \]

        \item[(ii)] 
        $\textstyle\lim_{p \to \infty} |\Pi^\ast_{p, \pm} \, f| = |\Pi^\ast_{\infty, \pm} \, f|$ 
        and
        $\textstyle\lim_{p \to \infty} \vtil_q(K, \Pi^{\ast}_{p, \pm} \, f) = \vtil_q(K, \Pi^{\ast}_{\infty, \pm} \, f)$
        ;

        \item[(iii)] 
        $\textstyle\lim_{s \to 1^-} |\Pi^{\ast, s}_{\infty, \pm} \, f| = |\Pi^\ast_{\infty, \pm} \, f|$ 
        and
        $\textstyle\lim_{s \to 1^-} \vtil_q(K, \Pi^{\ast, s}_{\infty, \pm} \, f) = \vtil_q(K, \Pi^{\ast}_{\infty, \pm} \, f)$.
    \end{itemize}
\end{proposition}

\begin{proposition}\label{prop.limit.1/p.pm}
    Let $f \in W^{1, 1}(\R^n) \cap W^{1, \infty}(\R^n)$ be nonzero and let  $K$ be a bounded star body.
	Then, the following limits hold: 
	\begin{itemize}
		\item[(i)] for any fixed $s \in (0, 1)$, one has $\textstyle\lim_{p \to \infty} \vtil_{-sp} (K, \Pi^{\ast, s}_{p, \pm} \, f)^{1/p} = \dila_s(K, \Pi^{\ast, s}_{\infty, \pm} \, f)$;
		
		\item[(ii)] $\textstyle\lim_{p \to \infty} \vtil_{-p}(K, \Pi^\ast_{p, \pm} \, f)^{1/p} = \dila(K, \Pi^\ast_{\infty, \pm} \, f)$;
		
		\item[(iii)] $\textstyle\lim_{s \to 1^-} \dila_s(K, \Pi^{\ast, s}_{\infty, \pm} \, f) = \dila(K, \Pi^\ast_{\infty, \pm} \, f)$.
	\end{itemize}
\end{proposition}

\subsection{Variants of $L^\infty$ P\'olya--Szeg\H{o} inequalities} 

As an application of our results, we establish variants of the P\'olya--Szeg\H{o}--type inequality in the $L^\infty$ framework associated with the $s$--fractional polar projection bodies.
This connects the anisotropic H\"older constant and the anisotropic $L^\infty$ norm of the gradient to rearrangement properties of functions, see e.g.~\cite[Corollary 8.2]{BS_2000} on the P\'olya--Szeg\H{o} inequality for the isotropic $L^\infty$--norm of the gradient.

\begin{corollary}\label{prop.PS-sup}
	Let $f \in W^{1, 1}(\R^n) \cap W^{1, \infty}(\R^n)$ be nonnegative and let
     $K$ be a bounded star body.
	Then, one has $f^\star \in W^{1, 1}(\R^n) \cap W^{1, \infty}(\R^n)$ and the following inequalities hold:
	\begin{itemize}
		\item[(i)] For any fixed $s \in (0, 1)$, one has
			\begin{align*}
				\sup_{x \neq y} \dfrac{(f(x) - f(y))_\pm}{\| x - y \|_K^s} \geq \sup_{x \neq y} \dfrac{(f^\star(x) - f^\star(y))_\pm}{\| x - y \|_{K^\star}^s}
			\end{align*}
			and
			\begin{equation}\label{PS-Holder}
				\sup_{x \neq y} \dfrac{|f(x) - f(y)|}{\| x - y \|_K^s} \geq \sup_{x \neq y} \dfrac{|f^\star(x) - f^\star(y)|}{\| x - y \|_{K^\star}^s}.
			\end{equation}
		
		\item[(ii)] One has
			\begin{align*}
				\sup_{\xi \in \Sphere} \dfrac{\| \langle \nabla f(\cdot), \xi \rangle_\pm \|_{L^\infty(\R^n)} }{\| \xi \|_K} \geq \sup_{\xi \in \Sphere} \dfrac{\| \langle \nabla f^\star(\cdot), \xi \rangle_\pm \|_{L^\infty(\R^n)} }{\| \xi \|_{K^\star}}
			\end{align*}
		and 
			\begin{align*}
				\sup_{\xi \in \Sphere} \dfrac{\| \langle \nabla f(\cdot), \xi \rangle \|_{L^\infty(\R^n)} }{\| \xi \|_K} \geq \sup_{\xi \in \Sphere} \dfrac{\| \langle \nabla f^\star(\cdot), \xi \rangle \|_{L^\infty(\R^n)} }{\| \xi \|_{K^\star}}.
			\end{align*}
	\end{itemize}

Equality in the above inequalities holds if $K$ is a centered ellipsoid and $f$ is a translate of $f^\star \circ \phi$ for some $\phi \in \mathrm{SL}(n)$.
\end{corollary}
%
%

\begin{proof}
Thanks to \cite[Theorem 8.2, Corollary 8.2]{BS_2000}, we have $f^\star \in W^{1, 1}(\R^n) \cap W^{1, \infty}(\R^n)$.

\medskip

\textit{(i)}
It suffices to prove the inequality
\begin{equation}\label{PS-sup-plus}
	\sup_{x \neq y} \dfrac{(f(x) - f(y))_+}{\| x - y \|_K^s} \geq \sup_{x \neq y} \dfrac{(f^\star(x) - f^\star(y))_+}{\| x - y \|_{K^\star}^s}.
\end{equation}
Once this is proved, the inequality \eqref{PS-Holder} can be obtained thanks to the following identity, for any bounded star by $L$
 	\[
 		\sup_{x \neq y} \dfrac{|f(x) - f(y)|}{\| x - y \|_L^s} = \max \left\{ \sup_{x \neq y} \dfrac{(f(x) - f(y))_+}{\| x - y \|_L^s}, \sup_{x \neq y} \dfrac{(f(x) - f(y))_-}{\| x - y \|_L^s} \right\}.
 	\]	
 	
\medskip

Recall that for each $p > 1$, the anisotropic P\'olya--Szeg\H{o} inequality for $L^p$ fractional Sobolev norms (see \cite[Theorem 11]{HL_2024}) leads to
\begin{equation}\label{PS-frac}
    \begin{split}
	\dfrac{n}{p(1 - s)}\vtil_{-sp}(K, \Pi^{\ast, s}_{p, +} \, f) = & ~ \int_{\R^n} \int_{\R^n} \dfrac{(f(x) - f(y))_+^p}{\| x - y \|_K^{n + ps}} \,dx\,dy \\
    \geq & ~ \int_{\R^n} \int_{\R^n} \dfrac{(f^\star(x) - f^\star(y))_+^p}{\| x - y \|_{K^\star}^{n + ps}} \,dx\,dy \\
    = & ~  
    \dfrac{n}{p(1 - s)}\vtil_{-sp}(K, \Pi^{\ast, s}_{p, +} \, f^\star).
    \end{split}
\end{equation}
Note that the anisotropic H\"older constant can be written in terms of  dilation factor
\begin{align*}
	\sup_{x \neq y} \dfrac{(f(x) - f(y))_+}{\| x - y \|_L^s} = \dila_s(L, \Pi^{\ast, s}_{\infty, +} \, f), \quad \text{ for every star body } L.
\end{align*}
Therefore, applying Proposition \ref{prop.limit.1/p.pm}--\textit{(i)} and letting $p \to \infty$ in the inequality~\eqref{PS-frac}, we immediately get~\eqref{PS-sup-plus}.

\medskip

\textit{(ii)} We simply just combine the inequalities in \textit{(i)} and Proposition \ref{prop.limit.1/p.pm}.

Equality case follows from the equality of~\eqref{PS-frac}, see e.g \cite[Theorem 11]{HL_2024}.
 This completes the proof of Proposition~\ref{prop.PS-sup}.
\end{proof}

To end this section, the next proposition establishes variants of the affine P\'olya--Szeg\H{o} inequality for the volume of (fractional) $L^\infty$ polar projection bodies.

\begin{corollary}
Let $f \in W^{1, 1}(\R^n) \cap W^{1, \infty}(\R^n)$ be nonnegative.
Then, the following inequalities hold:
	\begin{itemize}
		\item[(i)] For any fixed $s \in (0, 1)$, one has
			\begin{align*}
				| \Pi^{\ast,s}_{\infty, \pm} \, f|^{- \frac{s}{n}} \geq  | \Pi^{\ast,s}_{\infty, \pm} \, f^\star |^{- \frac{s}{n}} \quad \text{ and } \quad | \Pi^{\ast,s}_{\infty} \, f|^{- \frac{s}{n}} \geq  | \Pi^{\ast,s}_{\infty} \, f^\star |^{- \frac{s}{n}}.
			\end{align*}
		
		\item[(ii)] One has
			\begin{align*}
				| \Pi^{\ast}_{\infty, \pm} \, f|^{- \frac{1}{n}} \geq  | \Pi^{\ast}_{\infty, \pm} \, f^\star |^{- \frac{1}{n}} \quad \text{ and } \quad | \Pi^{\ast}_{\infty} \, f|^{- \frac{1}{n}} \geq  | \Pi^{\ast}_{\infty} \, f^\star |^{- \frac{1}{n}}.
			\end{align*}		
	\end{itemize}		

Equality in the above inequalities holds if $f$ is a translate of $f^\star \circ \phi$ for some $\phi \in \mathrm{SL}(n)$.
\end{corollary}

\begin{proof}
Recall that $f^\star \in W^{1, 1}(\R^n) \cap W^{1, \infty}(\R^n)$.

\medskip

\textit{(i)}
The inclusion~\eqref{inclu-01} implies that $f, f^\star \in W^{s, p}(\R^n)$ for every $s \in (0, 1)$ and $p \geq 1$.
Recall that the following inequality holds true (see \cite[Theorem 13--14]{HL_2024})
\begin{equation}\label{PS-vol}
    | \Pi^{\ast, s}_{p, \pm } \, f|^{- \frac{s}{n}} \geq |\Pi^{\ast, s}_{p, \pm} \, f^\star|^{-\frac{s}{n}} \quad \text{ and } \quad | \Pi^{\ast, s}_p \, f |^{- \frac{s}{n}} \geq |\Pi^{\ast, s}_p \, f|^{-\frac{s}{n}}, \quad \text{ for every $p > 1$}.
\end{equation}
Letting $p$ tend to $\infty$, thanks to Theorem~\ref{thm.volume} and Proposition~\ref{prop.lim.pm}, we obtain the desired inequalities.

\medskip

For similar reasons, we obtain the inequalities in \textit{(ii)} and the conclusion on the equality follows from the equality in \eqref{PS-vol}, which completes the proof.
\end{proof}

\subsection{Endpoint Lipschitz/Hölder isoperimetric-type inequalities}

For $1<p<n$, the affine Lutwak--Yang--Zhang inequality links the critical $L^{p^\ast}$ norm (with $p^\ast := \tfrac{np}{n-p}$) to the volume of the $L^p$ polar projection body of $f$ and it sharpens the classical $L^p$ Sobolev inequality.
More precisely, the following inequalities hold
\begin{equation}\label{LYZ-ineq}
    \| f \|^p_{L^\frac{np}{n - p}(\R^n)} \leq \alpha_{n, p} | \Pi^\ast_p \, f |^{-\frac{p}{n}} \leq \beta_{n, p} \| \nabla f \|_{L^p(\R^n)}^p, \quad \text{ for every } f \in W^{1, p}(\R^n),
\end{equation}
where the constants $\alpha_{n, p}$ and $\beta_{n, p}$ are sharp and have explicit formulas, see \cite{Aubin_1976, LYZ_2002, Talenti_1976}.
A fractional version of these inequalities has been recently proved by Haddad and Ludwig \cite{HL_2024}: for every $s \in (0, 1)$ and $p \in (1, n/s)$, it holds
\begin{equation}\label{HL-inequ}
    \| f \|^p_{L^{\frac{np}{n - ps}}(\R^n)} \leq \widehat\alpha_{n, p, s} | \Pi^{\ast, s}_p \, f|^{- \frac{ps}{n}} \leq \widehat\beta_{n, p, s} [f]^p_{s, p}, \quad \text{ for every } f \in W^{s, p}(\R^n),
\end{equation}
where the constants $\widehat\alpha_{n, p}$ and $\widehat\beta_{n, p}$ are sharp and have explicit formulas, see \cite{HL_2024, HL_2025}.
In \cite[Theorem~1]{HL_2024}, both parts of the affine fractional Sobolev inequality are stated under the condition $1 < p < n/s$. 
However, the restriction $p < n/s$ is only essential for the first (Sobolev-type) part.

\medskip

At the critical threshold $p=n$ in \eqref{LYZ-ineq} (or $p = n/s$ in \eqref{HL-inequ}) the Sobolev embedding breaks down, and for $p > n$ one only has Morrey embeddings into Hölder spaces.
Thus, there is no direct extension of the affine Sobolev inequalities to the endpoint $p=\infty$. 
Nevertheless, with a careful inspection, one can see that the second inequality in \eqref{LYZ-ineq} (respectively \eqref{HL-inequ}), which represents a functional affine $L^p$ isoperimetric inequality, holds for $p > n$ (respectively $p > n/s$).
In what follows, we introduce endpoint ($p=\infty$) isoperimetric-type inequalities formulated via the anisotropic H\"older constant and the associated fractional $L^\infty$ polar projection body $\Pi^{\ast,s}_{\infty} \, f$, which reduce to $\|\nabla f\|_{L^\infty}$ and $\Pi^{\ast}_{\infty} \, f$ at $s = 1$.

\medskip

Let us first recall the dual mixed volume inequality (see \cite[Section 9.5]{S_2014} and \cite[B.29]{G_2006}):
for any $q < 0$, it holds 
\begin{equation}\label{dual-mixed-ineq}
    \vtil_q(U, V) \geq |U|^{\frac{n - q}{n}} |V|^{\frac{q}{n}}, \quad \text{ for every star bodies } U, V.
\end{equation}
Equality holds if and only if $U$ and $V$ are dilates, i.e, there exists $c > 0$ such that $\rho_U  = c \rho_V$ on $\Sphere$.

\begin{corollary}[Endpoint isoperimetric--type inequality]\label{prop.endpoint}
Let $K$ be a bounded star body.
Then, for any nonzero function $f \in W^{1, 1}(\R^n) \cap W^{1, \infty}(\R^n)$, it holds
\begin{equation}\label{eq:endpoint-H}
\sup_{x \neq y} \dfrac{|f(x) - f(y)|}{\| x - y \|^s_K} \geq |K|^{\frac{s}{n}}|\Pi^{\ast, s}_\infty \, f|^{- \frac{s}{n}}, \quad \text{ for every } s \in (0, 1), 
\end{equation}
and
\begin{equation}\label{eq:endpoint-L}
\sup_{\xi \in \Sphere} \dfrac{\| \langle \nabla f(\cdot), \xi \rangle \|_{L^\infty(\R^n)}}{\| \xi \|_K} \geq
|K|^{\frac{1}{n}} |\Pi^\ast_\infty \, f|^{-\frac{1}{n}}.
\end{equation}
There is equality in~\eqref{eq:endpoint-H} and \eqref{eq:endpoint-L} if $\Pi^{\ast, s}_\infty \, f$ and $\Pi^\ast_\infty \, f$ is a dilate of $K$, respectively.
\end{corollary}

\begin{remark}\normalfont $ \, $
    \begin{itemize}
        \item[(1)]
        Analogous inequalities to \eqref{eq:endpoint-H}--\eqref{eq:endpoint-L} hold true for asymmetric fractional $L^p$ polar projection bodies $\Pi^{\ast, s}_{\infty, \pm} \, f$  and $\Pi^{\ast}_{\infty, \pm} \, f$.

        \item[(2)]
        In the case $K = B^n$, equality holds in~\eqref{eq:endpoint-H} and \eqref{eq:endpoint-L} if $f$ is radially symmetric. 
        This follows from the equality case of \eqref{iso-p}, see \cite[Theorem 1]{HL_2024}.
    \end{itemize} 
\end{remark}

\begin{proof}[Proof of Corollary~\ref{prop.endpoint}]
Fix $s \in (0, 1)$.
Applying the inequality~\eqref{dual-mixed-ineq} to the case $U = K$, $V = \Pi^{\ast, s}_p f$ and $q = - sp$, we have
\begin{equation}\label{iso-p}
    \vtil_{-sp}(K, \Pi^{\ast, s}_p f)^{\frac{1}{p}} \geq |K|^{\frac{1}{p} + \frac{s}{n}} |\Pi^{\ast, s}_p  f|^{- \frac{s}{n}}.
\end{equation}
Letting $p$ tend to $\infty$, thanks to Theorem~\ref{thm.mixed-dila}--\textit{(i)}, Theorem~\ref{thm.volume}--\textit{(i)} and the identity~\eqref{dila-holder}, we infer that
\begin{align*}
    \dila_s(K, \Pi^{\ast, s}_\infty  f) = \sup_{x \neq y} \dfrac{|f(x) - f(y)|}{\| x - y \|_K^s} \geq |K|^{\frac{s}{n}}|\Pi^{\ast, s}_\infty  f|^{- \frac{s}{n}}.
\end{align*}
Now, in the inequality~\eqref{eq:endpoint-H}, letting $s \to 1^-$ and using Theorem~\ref{thm.mixed-dila}--\textit{(iii)} and Theorem~\ref{thm.volume}--\textit{(iii)}, we obtain
\begin{align*}
    \sup_{\xi \in \Sphere} \dfrac{\| \langle \nabla f(\cdot), \xi \rangle \|_{L^\infty(\R^n)}}{\| \xi \|_K} = \dila(K, \Pi^{\ast}_\infty \, f) \geq |K|^{\frac{1}{n}} |\Pi^\ast_\infty \, f|^{-\frac{1}{n}}.
\end{align*}

\medskip

If the equality in \eqref{iso-p} holds, then for each $p > 1$, there exists $c_{s, p} > 0$ such that 
\[
	\| \xi \|_{\Pi^{\ast, s}_p f} = c_{s, p} \| \xi \|_K, \quad \text{ for every } \xi \in \Sphere.
\]
It follows from Proposition~\ref{prop.exist.infi.polar}--\textit{(i)} and Remark~\ref{rem.bounded-gauge}--\textit{(i)}, there exist $m_0, m_1 > 0$ independent of $p$ such that $m_0 \leq c_{s, p} \leq m_1$.
Hence, there exists $\overline c_{s, p} > 0$ such that, up to a subsequence, $c_p \to \overline c_s$ as $p \to \infty$. 
Therefore, using Proposition~\ref{prop.lim.gauge}--\textit{(i)}, we obtain
\begin{equation}\label{equa-endpoint}
	\| \xi \|_{\Pi^{\ast, s}_\infty f} = \overline c_s \| \xi \|_K \quad \text{ for a.e } \xi \in \Sphere,
\end{equation}
and the continuity of $\| \cdot \|_{\Pi^{\ast, s}_\infty \, f}$ (see Proposition~\ref{prop.exist.infi.polar}--\textit{(ii)}) implies that \eqref{equa-endpoint} holds true for every $\xi \in \Sphere$.
We finally conclude that $\Pi^{\ast, s}_\infty \, f$ is a dilate of $K$.

\medskip

Assume that for each $s \in (0, 1)$, the equality \eqref{equa-endpoint} holds for some $\overline c_s > 0$, which leads to the equality in \eqref{eq:endpoint-H} and, letting $s \to 1^-$, the equality of \eqref{eq:endpoint-L}.
On the one hand, it follows from Proposition~\ref{prop.exist.infi.polar}--\textit{(ii)} that there exists a constant $\overline m_0 > 0$ such that $\overline c_s \geq \overline m_0$ for every $s$ close to $1$.
 On the other hand, in view of Remark~\ref{rem.bounded-gauge}--\textit{(ii)}, we have $\| \xi \|_{\Pi^{\ast, s}_\infty f} \leq 1 + M^2$ for every $s \in (1/2, 1)$ and $\xi \in \Sphere$.
Hence, there exists a constant $\overline m_1  > 0$ such that $\overline c_s \leq \overline m_1$ for every $s \in (1/2, 1)$.
Therefore, up to a subsequence, we infer that $\overline c_s \to \overline c$ as $s \to 1^-$ for some $\overline c > 0$.
Combining this fact with Proposition~\ref{prop.lim.gauge}--\textit{(iii)} and the continuity of gauge functions yields
\[
	\| \xi \|_{\Pi^\ast_\infty \, f} = \overline c \| \xi \|_K \quad \text{ for every } \xi \in \Sphere,
\]
so $\Pi^\ast_\infty \, f$ is a dilate of $K$.
Corollary~\ref{prop.endpoint} is proven.
\end{proof}

\noindent\rule{5cm}{1pt} \smallskip\newline\noindent\textbf{Acknowledgements.}
The author would like to thank Aris Daniilidis, Alberto Dom\'inguez Corella and Sebasti\'an Tapia-Garc\'ia for fruitful discussions and comments that helped to improve the presentation of the present manuscript.
This research was funded in part by the Austrian Science Fund (FWF) [DOI 10.55776/P36344N].

\noindent Tr\'i Minh L\^E

\medskip

\noindent Institut f\"{u}r Stochastik und Wirtschaftsmathematik, VADOR E105-04
\newline TU Wien, Wiedner Hauptstra{\ss }e 8, A-1040 Wien\medskip
\newline\noindent E-mail: \texttt{minh.le@tuwien.ac.at}
\newline\noindent\texttt{https://sites.google.com/view/tri-minh-le} \smallskip\newline
\noindent Research supported by the FWF (Austrian Science Fund) grant \texttt{DOI 10.55776/P-36344N}.
\end{document}